\documentclass[a4paper,11pt]{article}
\usepackage[utf8]{inputenc}
\usepackage[T1]{fontenc}
\usepackage{enumitem}
\usepackage[english]{babel}
\usepackage{amsmath, amssymb, mathrsfs, amsfonts, stmaryrd, wasysym}
\usepackage{geometry}
\usepackage{titlesec}
\usepackage{xcolor}
\usepackage{authblk}
\usepackage{graphicx}
\usepackage{array}
\usepackage{subcaption}
\usepackage{changepage}
\usepackage{hyperref}
\hypersetup{
    colorlinks=true,
    linkcolor=blue,
    citecolor=blue
}

\usepackage{amsthm}
\usepackage{thmtools}

\usepackage{tikz}
\usepackage{tikz-cd}

\usepackage{chngcntr}

\titlespacing{\section}{0pt}{3.5ex plus 1ex minus .2ex}{2.3ex plus .2ex}
\titlespacing{\subsection}{0pt}{3.25ex plus 1ex minus .2ex}{1.5ex plus .2ex}
\titlespacing{\subsubsection}{0pt}{3.25ex plus 1ex minus .2ex}{1.5ex plus .2ex}

\titleformat{\section}
{\normalfont\Large\bfseries}{\thesection}{1em}{}[\vspace{0.5cm}]
\titleformat{\subsection}
{\normalfont\large\bfseries}{\thesubsection}{1em}{}[\vspace{0.3cm}]
\titleformat{\subsubsection}
{\normalfont\normalsize\bfseries}{\thesubsubsection}{1em}{}[\vspace{0.2cm}]

\declaretheoremstyle[
  spaceabove=14pt, spacebelow=14pt,
  headfont=\normalfont\bfseries,
  notefont=\normalfont,
  bodyfont=\normalfont,
  postheadspace=1em,
  headpunct={}
]{mystyle}

\declaretheorem[style=mystyle,numberwithin=section,name=Theorem]{thm}
\declaretheorem[style=mystyle,sibling=thm,name=Lemma]{lem}
\declaretheorem[style=mystyle,sibling=thm,name=Proposition]{prop}
\declaretheorem[style=mystyle,sibling=thm,name=Definition-Proposition]{def-prop}
\declaretheorem[style=mystyle,sibling=thm,name=Remark]{rem}

\declaretheorem[style=mystyle,sibling=thm,name=Definition]{defi}
\declaretheorem[style=mystyle,numbered=no,name=Definition]{defi*}

\declaretheorem[style=mystyle,sibling=thm,name=Conjecture]{conj}

\title{Simple-stable representations of surface groups in $\mathrm{PU}(2,1)$}
\author[1]{Ulysse Remfort-Aurat}
\affil[1]{Aix-Marseille Université, CNRS, I2M, Marseille, France\\ \texttt{ulysse.remfort-aurat@univ-amu.fr}}
\date{}

\newcommand{\C}{\mathbb{C}}
\newcommand{\R}{\mathbb{R}}

\newcommand{\h}{\mathbb{H}^2_{\mathbb{C}}}
\newcommand{\Or}{\tau_{\rho,x}}

\begin{document}

\maketitle

\begin{abstract}
Let $\Gamma_g$ be the fundamental group of a closed orientable surface of genus $g\geqslant 2$. The outer automorphism group $\mathrm{Out}(\Gamma_g)$ naturally acts on the character variety $\mathcal{X}(\Gamma_g,G)$ for any Lie group $G$. We consider the set of simple-stable representations which are modelled on Minsky's primitive-stable representations. We prove that the set of conjugacy classes of simple-stable representations of $\Gamma_g$ in $\mathrm{PU}(2,1)$ is a domain of discontinuity for this action, strictly larger than the set of conjugacy classes of convex cocompact representations. 
\end{abstract}

\tableofcontents

\newpage 

\section*{Introduction}

Let $\Sigma_g$ be a closed, oriented surface of genus $g\geqslant 2$ and $\Gamma_g=\pi_1(\Sigma_g)$ denote its fundamental group. When $\Sigma_g$ is endowed with a hyperbolic structure, that is, a Riemannian metric of constant curvature $-1$, its universal cover $\widetilde{\Sigma_g}$ is isometric to the hyperbolic plane $\mathbb{H}^2$, whose group of orientation-preserving isometries is isomorphic to $\mathrm{PSL}_2(\mathbb{R})$. Consequently, $\Sigma_g$ itself is isometric to the quotient of $\mathbb{H}^2$ by a discrete subgroup of $\mathrm{PSL}_2(\mathbb{R})$, isomorphic to $\Gamma_g$.\\

A choice of specific generators for $\Gamma_g$ is called a \textit{marking}. The space of homotopy classes of marked hyperbolic structures on $\Sigma_g$ is called the Fricke space and is denoted by $\mathcal{F}(\Sigma_g)$. It is homeomorphic to an open ball of real dimension $6g - 6$ and admits two natural embeddings (one for each orientation on $\Sigma_g$) into $\mathrm{DF}\bigl(\Gamma_g, \mathrm{PSL}_2(\mathbb{R})\bigl)$, the space of conjugacy classes of discrete and faithful representations of $\Gamma_g$ into $\mathrm{PSL}_2(\mathbb{R})$.\\

Given a representation $\rho:\Gamma_g \to \mathrm{PSL}_2(\mathbb{R})$, we define an \textit{orbit map} to be a $\rho$-equivariant map $\tau_{\rho} : C(\Gamma_g) \to \mathbb{H}^2$ from a Cayley graph of $\Gamma_g$ to $\mathbb{H}^2$, which sends $\gamma \in \Gamma_g$ to $\rho(\gamma)(x)$ for a chosen basepoint $x \in \mathbb{H}^2$. When $\rho$ arises from a hyperbolic structure on $\Sigma_g$, it follows from the \v{S}varc-Milnor lemma that any orbit map is a quasi-isometry. In fact, this property characterizes discrete and faithful representations of $\Gamma_g$ into $\mathrm{PSL}_2(\mathbb{R})$.\\

The set $\mathrm{DF}\bigl(\Gamma_g, \mathrm{PSL}_2(\mathbb{R})\bigl)$ is naturally embedded in the character variety $\mathcal{X}\bigl(\Gamma_g, \mathrm{PSL}_2(\mathbb{R})\bigl)$, that is, the set of conjugacy classes of representations of $\Gamma_g$ into $\mathrm{PSL}_2(\mathbb{R})$. Goldman~\cite{goldman1988topological} proved that this larger space has $4g - 3$ connected components, distinguished by the Euler number of an associated $S^1$-bundle over $\Sigma_g$. Moreover, he showed that $\mathrm{DF}(\Gamma_g, \mathrm{PSL}_2(\mathbb{R}))$ is precisely the union of the two connected components corresponding to the extreme values of the Euler number.\\

The extended mapping class group of $\Sigma_g$, denoted by $\mathrm{MCG}(\Sigma_g)$, is the group of isotopy classes of diffeomorphisms of $\Sigma_g$. Any such class induces an outer automorphism of $\Gamma_g$, and the Dehn-Nielsen-Baer theorem states that $\mathrm{MCG}(\Sigma_g)$ is isomorphic to $\mathrm{Out}(\Gamma_g)$, where $\mathrm{Out}(\Gamma_g) = \mathrm{Aut}(\Gamma_g)/\mathrm{Inn}(\Gamma_g)$ is the outer automorphism group of $\Gamma_g$.

The index two subgroup consisting of isotopy classes of orientation-preserving diffeomorphisms of $\Sigma_g$ acts naturally on the space of marked hyperbolic structures. A classical theorem due to Fricke asserts that this action is properly discontinuous. Consequently, $\mathrm{Out}(\Gamma_g)$ acts properly discontinuously on $\mathrm{DF}(\Gamma_g, \mathrm{PSL}_2(\mathbb{R}))$. The dynamics on the complementary subset of $\mathrm{DF}(\Gamma_g, \mathrm{PSL}_2(\mathbb{R}))$ remain poorly understood, but were conjectured by Goldman~\cite{goldman2005mapping} to be ergodic. Marché-Wolff~\cite{marche2016modular} proved this conjecture in the case where the genus $g$ is equal to $2$ but the question remains open for higher genuses.\\

A natural generalization is to consider (conjugacy classes of) representations of $\Gamma_g$ into larger Lie groups $G$ containing $\mathrm{PSL}_2(\mathbb{R})$. The Fricke space also embeds into the character variety $\mathcal{X}(\Gamma_g, G)$ and provides an $\mathrm{Out}(\Gamma_g)$-invariant subset on which the action is properly discontinuous. A first step toward understanding the global dynamics of $\mathrm{Out}(\Gamma_g)$ on $\mathcal{X}(\Gamma_g, G)$ is to identify maximal \textit{domains of discontinuity}, that is, maximal open $\mathrm{Out}(\Gamma_g)$-invariant subsets of $\mathcal{X}(\Gamma_g, G)$ on which the action is properly discontinuous.\\

We first consider the case where $G$ has rank one. These groups are essentially the isometry groups of hyperbolic spaces, for instance :
\begin{samepage}
\begin{itemize}
\item $G = \mathrm{PO}(n,1)$ is the group of isometries of the $n$-dimensional real hyperbolic space $\mathbb{H}^n_{\mathbb{R}}$,
\item $G = \mathrm{PU}(n,1)$ is the group of holomorphic isometries of the $n$-dimensional complex hyperbolic space $\mathbb{H}^n_{\mathbb{C}}$,
\item $G = \mathrm{Sp}(n,1)$ is the group of isometries of the $n$-dimensional quaternionic hyperbolic space,
\item $G = F_{4(-20)}$ is the group of isometries of the octonionic hyperbolic plane.
\end{itemize}
\end{samepage}

In all these cases, Fricke's theorem can be adapted to show that the set $\mathrm{CC}(\Gamma_g, G)$ of conjugacy classes of convex cocompact representations is a domain of discontinuity in the character variety $\mathcal{X}(\Gamma_g, G)$. We say that a representation is \textit{convex cocompact} if it admits an orbit map that is a quasi-isometric embedding. Note that when $G \neq \mathrm{PSL}_2(\mathbb{R})$, not every discrete and faithful representation of $\Gamma_g$ into $G$ is convex cocompact.\\

An extensively studied case is when $G = \mathrm{Isom}^+(\mathbb{H}_{\mathbb{R}}^3) \simeq \mathrm{PSL}_2(\mathbb{C}) \simeq \mathrm{PO}_0(3,1)$ is the group of orientation preserving isometry of $\mathbb{H}_{\mathbb{R}}^3$. Goldman~\cite{goldman1988topological} proved that the space $\mathcal{X}(\Gamma_g, \mathrm{PSL}_2(\mathbb{C}))$ has two connected components. Every convex cocompact representation lies in the same connected component. The set $\mathrm{CC}(\Gamma_g, \mathrm{PSL}_2(\mathbb{C}))$ can be identified with the product of two copies of the Fricke space via Bers' Simultaneous Uniformization Theorem~\cite{Bers}. It is conjectured that $\mathrm{CC}(\Gamma_g, \mathrm{PSL}_2(\mathbb{C}))$ is a maximal domain of discontinuity. Once again, the dynamics on the complement are not well understood, but are expected to be ergodic.\\

Another interesting case is when $G=\mathrm{PU}(2,1) \simeq \mathrm{Isom}^{hol}(\mathbb{H}^2_{\mathbb{C}})$ is the group of holomorphic isometries of the complex hyperbolic plane $\mathbb{H}_{\mathbb{C}}^2$. Xia~\cite{Xia} proved that $\mathcal{X}(\Gamma_g, \mathrm{PU}(2,1))$ has $6g - 5$ connected components, distinguished by the value of the Toledo invariant. Toledo~\cite{Toledo} proved that every representation with extremal Toledo invariant is convex cocompact. Goldman, Kapovich, and Leeb~\cite{GKL}, as well as Bronstein~\cite{Bronstein}, proved that many other connected components of $\mathcal{X}(\Gamma_g, \mathrm{PU}(2,1))$ also contain convex cocompact representations. The main result of this paper shows that $\mathrm{CC}(\Gamma_g, \mathrm{PU}(2,1))$ is not a maximal domain of discontinuity.\\

The larger domain of discontinuity constructed in this article is modeled on Minsky's primitive-stable representations. Let $\Gamma$ be a torsion-free hyperbolic group. Every nontrivial element $\gamma \in \Gamma$ acts on the Cayley graph $C(\Gamma)$ with two fixed points $\gamma_-$ and $\gamma_+$ on its Gromov boundary. 

If $A \subset \Gamma$, we say that a representation $\rho : \Gamma \to G$ is $A$-stable if there exist uniform constants $K \geq 1$, $C \geq 0$, and an orbit map $\tau$ such that the restriction of $\tau$ to any geodesic in $C(\Gamma)$ joining the two fixed points of an element of $A$ is a $(K,C)$-quasi-isometric embedding. This definition naturally generalizes convex-cocompactness since a representation is convex cocompact if and only if it is $\Gamma$-stable.\\

Minsky first studied the case where $\Gamma = F_k$ is the free group of rank $k$, $A$ is the set of primitive elements of $F_k$, and $G = \mathrm{PSL}_2(\mathbb{C})$. He proved that the set of primitive-stable representations $S_{\mathcal{P}}(F_k, \mathrm{PSL}_2(\mathbb{C}))$ is a domain of discontinuity containing $\mathrm{CC}(F_k, \mathrm{PSL}_2(\mathbb{C}))$, as well as elements in $\partial \mathrm{CC}(F_k, \mathrm{PSL}_2(\mathbb{C}))$ and conjugacy classes of representations that are not discrete and faithful.\\

When $\Gamma$ is the fundamental group of a closed surface $\Sigma$, we say that a nontrivial element of $\Gamma$ is \textit{simple} if it corresponds to the free homotopy class of a closed curve without self-intersections. We denote by $S_{\mathcal{S}}(\Gamma, G)$ the set of conjugacy classes of simple-stable representations.\\

On the one hand, Lee proved a result analogous to Minsky's by considering the set of simple-stable representations of the fundamental group of a non-orientable surface with negative Euler characteristic into $\mathrm{PSL}_2(\mathbb{C})$. On the other hand, she also showed that in the orientable case, the set $S_{\mathcal{S}}(\Gamma_g, \mathrm{PSL}_2(\mathbb{C}))$ of simple-stable representations does not intersect $\partial \mathrm{CC}(\Gamma_g, \mathrm{PSL}_2(\mathbb{C}))$.\\

Our main result shows that the situation differs significantly when $G = \mathrm{PU}(2,1)$.

\begin{restatable}{Theoreme}{SSD}\label{SSD}
Let $\Gamma_g$ denote the fundamental group of a closed oriented surface of genus $g \geq 2$.\\
The set $S_{\mathcal{S}}(\Gamma_g, \mathrm{PU}(2,1))$ of conjugacy classes of simple-stable representations of $\Gamma_g$ into $\mathrm{PU}(2,1)$ is a domain of discontinuity strictly larger than $\mathrm{CC}(\Gamma_g, \mathrm{PU}(2,1))$. Moreover, $S_{\mathcal{S}}(\Gamma_g, \mathrm{PU}(2,1))$ contains elements of $\partial \mathrm{CC}(\Gamma_g, \mathrm{PU}(2,1))$, as well as conjugacy classes of representations that are not discrete and faithful.
\end{restatable}

This theorem confirms the intuition of Parker and Platis in~\cite[Problem~6.2 and the subsequent discussion]{PP} by showing that there exist discrete and faithful representations on the boundary $\partial \mathrm{CC}(\Gamma_g, \mathrm{PU}(2,1))$ whose images contain parabolic elements that do not correspond to simple closed curves.\\

This result also contrasts with conjectures made by Bowditch~\cite{Bow4}. He studied representations satisfying a property implied by simple-stability. When $G=\mathrm{Isom}(X)$ is the isometry group of a metric space $X$, we say that a representation $\rho\in \mathcal{R}(\Gamma,G)$ is simple-well-displacing if there exists uniform constants $J\geqslant 1$ and $B\geqslant 0$ such that:
$$\forall \gamma \in \mathcal{S}, \frac{1}{J}l(\rho(\gamma))-B\leqslant l_S(\gamma) \leqslant Jl(\rho(\gamma))+B $$ 
where $l(\rho(\gamma)$ and $l(\gamma)$ denote the translation length of $\rho(\gamma)$ on $X$ and of $\gamma$ on a Cayley graph $C(\Gamma)$ respectively.\\

Bowditch~\cite[Questions B and C]{Bow4} asked the following questions:
\begin{itemize}
\item If $\rho\in \mathcal{R}\bigl(\Gamma_g,\mathrm{PSL}_2(\mathbb{R})\bigl)$ is a representation such that the image of every simple element is hyperbolic, must $\rho$ be discrete and faithful?
\item If $\rho\in \mathcal{R}\bigl(\Gamma_g,\mathrm{PSL}_2(\mathbb{C})\bigl)$ is a simple-well-displacing representation, must $\rho$ be discrete and faithful?
\end{itemize}

Theorem \ref{SSD} proves that there exists simple-well-displacing representations of $\Gamma_g$ in $\mathrm{PU}(2,1)$ which are not discrete and faithful, highlighting another constrast between the real and complex cases.\\

Another important setting arises when $G$ is a higher-rank Lie group. In this case, the class of convex cocompact representations is too rigid, and one instead considers Anosov representations, which form a domain of discontinuity for the action of $\mathrm{Out}(\Gamma_g)$ on $\mathcal{X}(\Gamma_g, G)$. Wang~\cite{wang2023anosov} generalized the notion of $A$-stable representations to this setting. Tholozan-Wang~\cite{tholozan2023simple} studied representations of $\Gamma_g$ into $\mathrm{PSL}_{2d}(\mathbb{C})$ for $d \geq 2$ and proved that the set of conjugacy classes of \textit{simple-Anosov representations} forms a domain of discontinuity strictly containing the set of Anosov representations. 

The representations constructed in this article provide examples of simple-Anosov representations of $\Gamma_g$ into $\mathrm{PU}(2,1) \subset \mathrm{PSL}_3(\mathbb{C})$ that are not Anosov, thereby answering the second part of Question~1.6 in~\cite{tholozan2023simple}.\\

Along the way, we prove the following characterization of simple-stable representations among the discrete, faithful and geometrically finite ones, that we believe to be interesting on its own right. A \textit{pinched Hadamard manifold} is a complete, simply-connected manifold $X$ whose real sectional curvature lies between two negative real numbers. A discrete subgroup $H$ of $G$ is said geometrically finite if its limit set consist entirely of "conical limit points" and "bounded parabolic fixed points" (definitions are given in Section \ref{GF}).

\begin{restatable}{Theoreme}{Crit}\label{Crit}
Let $\Gamma_g=\pi_1(\Sigma_g)$ be the fundamental group of a closed surface $\Sigma_g$ of genus $g\geqslant 2$, $G$ be the isometry group of a pinched Hadamard manifold and let $\rho\in \mathcal{R}(\Gamma_g,G)$ be a discrete, faithful and geometrically finite representation.\\
The representation $\rho$ is simple-stable if and only if the image of every simple element of $\Gamma_g$ is hyperbolic.
\end{restatable}

Another related 

\textbf{Organization of the paper.} In Section 2 and 3, we give some preliminary results about hyperbolic spaces and geometrical finiteness respectively. In Section 4, we define $A$-stable representations, prove that their conjugacy classes form a domain of discontinuity and give another useful characterization. In Section 5, we specialize to simple-stable representations of surface groups and prove Theorem \ref{Crit}. Section 6 deals with the complex hyperbolic space $\mathbb{H}^2(\mathbb{C})$ and its group of holomorphic isometries $\mathrm{PU}(2,1)$. We define Dirichelt fundamental domains and prove that a subgroup of $\mathrm{PU}(2,1)$ which admits a finite sided Dirichelt fundamental domain is geometrically finite. In Section 7, we deal with triangular groups and study their deformations in $\mathrm{PU}(2,1)$. Finally, in Section 8, we gather all the previous material to prove Theorem \ref{SSD}.

\section{Preliminaries}

\subsection{Gromov-hyperbolic metric spaces and their boundaries}

Let $(X,d)$ be a metric space and $x,y,o\in X$. The \textit{Gromov product} of $x$ and $y$ based at $o$ is denoted by $\langle x,y \rangle_o$. It is defined by: 

$$\langle x,y \rangle_o=\frac{1}{2}\bigl(d(x,o)+d(y,o)-d(x,y)\bigl).$$

\begin{defi}
The metric space $X$ is \textit{Gromov-hyperbolic} if there is $\delta>0$ such that for all $o,x,y,z$ in $X$, we have: 

$$\langle x,z \rangle_o \geqslant \min \left\{\langle x,y \rangle_o,\langle y,z \rangle_o\right\} - \delta.$$
\end{defi}

We call such a $\delta$ a \textit{hyperbolicity constant} for $X$. If we want to insist on that constant, we say that $X$ is a $\delta$-hyperbolic space. The metric space $X$ is \textit{geodesic} if for each pair $x,y\in X$, there is a geodesic segment between $x$ and $y$, i.e., an isometric embedding $f:[0,d(x,y)] \rightarrow X$ such that $f(0)=x$ and $f(d(x,y))=y$.\\

Assume that the space $X$ is Gromov-hyperbolic and geodesic and let $o\in X$ be a basepoint.\\
In that context, the Gromov product $\langle x,y \rangle_o$ is roughly equal to the distance between $o$ and a geodesic segment with endpoints $x$ and $y$ (see e.g. \cite[Lemma 2.33]{Vai}). We say that a sequence $(x_n)_{n\in \mathbb{N}}$ of elements of $X$ \textit{converges at infinity} if $\langle x_n , x_m \rangle_o \rightarrow \infty$ as $n,m \rightarrow \infty$. Two such sequences $(x_n)_{n\in \mathbb{N}}$ and $(y_n)_{n\in \mathbb{N}}$ are said to be \textit{equivalent} if $\langle x_n, y_m \rangle_o \rightarrow \infty$ as $n,m \rightarrow \infty$. These definitions are independent of the choice of the basepoint $o\in X$.\\

The set of equivalence classes of sequences converging at infinity is called the \textit{Gromov boundary} of $X$ and denoted by $\partial X$. The Gromov product extends to $\partial X$. Let $a,b\in \partial X$, we define: 

$$\langle a,b\rangle_o=\sup \liminf_{n,m \rightarrow \infty} \langle x_n,y_m \rangle_o$$

where the supremum is taken over all sequences $(x_n)_{n\in \mathbb{N}}$ and $(y_m)$ in $X$ whose equivalence class is $a$ and $b$ respectively.\\

A neighbourhood system for $a \in \partial X$ is given by the sets: 
$$\{b\in X\cup \partial X, \langle a,b\rangle_o \geqslant R\} ~~~~~~\mathrm{for}~~ R\geqslant 0.$$ 
The topological spaces $\partial X$ and $\overline{X}= X \cup \partial X$ are Hausdorff (see e.g. \cite[Section 4 and 5]{Vai}). If $X$ is \textit{proper}, i.e., closed balls are compact, then $\partial X$ and $\overline{X}$ are compact. We call $\overline{X}$ the \textit{Gromov compactification} of $X$.\\

\subsection{Isometry group and convergence action}

Let $G=\mathrm{Isom}(X)$ be the isometry group of a Gromov-hyperbolic and geodesic metric space $X$. We endow $G$ with the compact-open topology.\\

Let $g$ be an element of $G$ and $o\in X$ be a basepoint. The \textit{translation length of $g$} is denoted by $l(g)$. It is defined by: $$l(g)=\inf_{x\in X} d\bigl(x,g(x)\bigl).$$
The \textit{stable norm of $g$} is denoted by $N(g)$. It is defined by:
$$N(g)=\lim_{n\to +\infty}\frac{d\bigl(o,g^n(o)\bigl)}{n}.$$

The stable norm is well-defined by a subadditivity argument and the limit does not depend on the choice of the basepoint $o\in X$. Two conjugate elements of $G$ have the same translation length and stable norm.
When $X$ is $\delta$-hyperbolic, the translation length and stable norm only differ by $16\delta$ (see {\cite[Chapter 10, Proposition 6.4]{coornaert2006geometrie}}).\\

An isometry of $X$ extends to a homeomorphism of $\overline{X}$. Elements of $G$ fall into one of the following three categories: 
\begin{itemize}
\item $g$ is \textit{elliptic} if $N(g)=0$ and $g$ has a bounded orbit.
\item $g$ is \textit{parabolic} if $N(g)=0$ and $g$ is not elliptic. In that case, the extended action of $g$ has a unique fixed point in $\partial X$.
\item $g$ is \textit{hyperbolic} if $N(g)>0$. In that case, the extended action of $g$ has exactly two distinct fixed points in $\partial X$.
\end{itemize}

If $X$ is also proper, Tukia proved that the isometric action of $G$ on $X$ extends to a \textit{convergence action}, as defined in \cite{gehring1987discrete} by Gehring and Martin, on the Gromov compactification $\overline{X}$.

\begin{thm}[{\cite[Theorem 3A]{tukia1994convergence}}]\label{Tukia}
Let $G$ be the isometry group of a Gromov-hyperbolic, geodesic and proper metric space $X$ and $(g_n)_{n\in \mathbb{N}}$ be a sequence of elements of $G$. Exactly one of the two following statement holds:

\begin{itemize}
\item The sequence $(g_n)_{n\in \mathbb{N}}$ is relatively compact.
\item Up to subsequence, there exist $a,b\in \partial X$ such that the restriction of $(g_n)_{n\in \mathbb{N}}$ to $\overline{X}\setminus \{a\}$ converges to $b$ uniformly on compact subsets.
\end{itemize}
\end{thm}

A sequence $(g_n)_{n\in\mathbb{N}}$ satisfying the second property is called a \textit{convergence sequence with base $(a,b)$}. Note that if such a pair $(a,b)$ exists, then it is unique and that one may have $a=b$. If a subgroup $\Gamma$ of $G$ is discrete then every sequence of distinct elements of $\Gamma$ has a convergence subsequence.

\subsection{Hyperbolic groups}

Let $\Gamma$ be a finitely generated group and $S$ be a finite generating set. The \textit{Cayley graph} of $\Gamma$ with respect to $S$ is denoted by $C_S(\Gamma)$. It is the graph whose vertices are labelled by elements of $\Gamma$, and there is an edge between two vertices labelled by $\gamma,\gamma'\in \Gamma$ if there is $s\in S$ such that $\gamma'=\gamma s$ or $\gamma'=\gamma s^{-1}$. The length metric on $C_S(\Gamma)$ obtained by setting each edge to be of length 1, yields a metric on both $C_S(\Gamma)$ and $\Gamma$. In both cases, we call it the \textit{word metric} and denote it by $d_S$. The group $\Gamma$ acts naturally by isometries on $(C_S(\Gamma),d_S)$ by left multiplication on vertices.\\

The \textit{word length of $\gamma\in \Gamma$} is $|\gamma|_S = d_S(e,\gamma)$, where $e$ is the identity element of $\Gamma$. The translation length of $\gamma\in\Gamma$ for its action on $C_S(\Gamma)$ is denoted by $l_S(\gamma)$. \\

A group $\Gamma$ is \textit{hyperbolic} if there is a finite generating set $S$ such that the Cayley graph $(C_S(\Gamma),d_S)$ is a Gromov-hyperbolic metric space. If $\Gamma$ is a torsion-free hyperbolic group, every non-trivial element of $\Gamma$ acts on $C_S(\Gamma)$ as a hyperbolic isometry. We will denote by $\partial \Gamma$ the Gromov compactification of $(C_S(\Gamma),d_S)$. Up to homeomorphism, $\partial \Gamma$ does not depend on the choice of the finite generating set $S$ for $\Gamma$.

\subsection{Space of representations}

Let $X$ be a metric space and let $G=\mathrm{Isom}(X)$ denote its isometry group endowed with the compact-open topology. Let $\Gamma$ be a finitely generated group and let $S$ be a finite generating set for $\Gamma$.\\

The \textit{space of representations} $\mathcal{R}(\Gamma,G)$ is the set of homomorphisms $\rho:\Gamma\longrightarrow G$. To define a topology on $\mathcal{R}(\Gamma,G)$, we first endow $G^{|S|}$ with the product topology and then identify $\mathcal{R}(\Gamma,G)$ with a subset of $G^{|S|}$. The subspace topology on $\mathcal{R}(\Gamma,G)$ does not depend on the choice of a finite generating set $S$.\\

The automorphism groups of $G$ and that of $\Gamma$ act on $\mathcal{R}(\Gamma,G)$: for every $\rho\in \mathcal{R}(\Gamma,G)$, $\psi\in \mathrm{Aut}(G)$ and $\varphi\in \mathrm{Aut}(\Gamma)$, define:
\begin{center}
$\psi \cdot \rho = \psi \circ \rho \qquad $ and $ \qquad \varphi \cdot \rho = \rho \circ \varphi^{-1}$.
\end{center}

Note that the actions of $\mathrm{Aut}(G)$ and $\mathrm{Aut}(\Gamma)$ commute. For a group $H$, we denote by $\mathrm{Inn}(H)$ the normal subgroup of inner automorphisms of $H$ and $\mathrm{Out}(H)=\mathrm{Aut}(H) / \mathrm{Inn}(H)$ the group of outer automorphisms of $H$.\\

We denote by $\mathcal{X}(\Gamma,G)=\mathcal{R}(\Gamma,G) / \mathrm{Inn}(G)$ the space of conjugacy classes of representations, and endow it with quotient topology. We mention that this topology may not be Hausdorff (it may not even be $T_1$), however, this fact will not add any difficulty to what follows.
The group $\mathrm{Inn}(\Gamma)$ acts trivially on $\mathcal{X}(\Gamma,G)$, hence there is an induced action of $\mathrm{Out}(\Gamma)$ on $\mathcal{X}(\Gamma,G)$.\\

We denote by $\mathrm{DF}(\Gamma,G)$ the set of conjugacy classes of faitfhul representations with discrete images. Note that the subset $\mathrm{DF}(\Gamma,G)$ is invariant under the action of $\mathrm{Out}(\Gamma)$.

\subsection{Orbit maps and convex cocompact representations}

Let $X$ be a metric space and let $G=\mathrm{Isom}(X)$ denote its isometry group endowed with the compact-open topology. Let $\Gamma$ be a finitely generated group and let $S$ be a finite generating set for $\Gamma$.\\

Let $\rho\in \mathcal{R}(\Gamma,G)$ be a representation, $x\in X$ be a basepoint and $S$ be a finite and symmetric generating set for $\Gamma$.\\

An \textit{orbit map} $\tau_{\rho,x}:C_S(\Gamma) \longrightarrow X$ is a $\rho$-equivariant map such that, for all $\gamma\in \Gamma$, $\tau_{\rho,x}(\gamma)= \rho(\gamma)(x)$ and $\tau_{\rho,x}$ sends each edge $[\gamma,\gamma s]$ of $C_S(\Gamma)$ to a geodesic segment of $X$ joining $\rho(\gamma)(x)$ to $\rho(\gamma s)(x)$. 

We recall the definition of quasi-isometric embeddings and quasi-geodesics.

\begin{defi}
Let $X$ and $Y$ be two metric spaces.

\begin{itemize}
\item A map $f:X \longrightarrow Y$ is a \textit{quasi-isometric embedding} if there are constants $K\geqslant 1$ and $C\geqslant 0$ such that for all $x,x'\in X$: 

$$\frac{1}{K}d_X(x,x')-C \leqslant d_Y\bigl(f(x),f(x')\bigl) \leqslant Kd_X(x,x')+C.$$

\vspace{0.1cm}

\item A \textit{quasi-geodesic} of $X$ is the image of a quasi-isometric embedding $f:\mathbb{R}\longrightarrow X$. 
\end{itemize}

If we want to insist on constants, we say that a map (respectively a path) is a $(K,C)$ quasi-isometric embedding (respectively a $(K,C)$ quasi-geodesic).

\end{defi}

A representation $\rho\in \mathcal{R}(\Gamma,G)$ is \textit{convex cocompact} if it admits an orbit map which is a quasi-isometric embedding. This definition depends neither on the choice of representation inside its conjugacy class, nor on the choice of the basepoint $x\in X$, nor on the choice of the finite generating set for $\Gamma$.\\

When $G$ is a rank one Lie group, $G$ acts by isometries on its symmetric space which is a geodesic, proper and Gromov-hyperbolic metric space $X$. In that setting, a representation $\rho\in \mathcal{R}(\Gamma,G)$ is convex cocompact if and only if it has finite kernel and there exists a convex and $\rho(\Gamma)$-invariant subset of $X$ on which the action is cocompact. 

We denote by $\mathrm{CC}(\Gamma,G)$ the set of conjugacy classes of convex cocompact representations. The subset $\mathrm{CC}(\Gamma,G)$ is invariant under the action of $\mathrm{Out}(\Gamma)$, and when $\Gamma$ is torsion-free, it is contained in $\mathrm{DF}(\Gamma,G)$.\\

\subsection{Geometrically finite subgroups}\label{GF}

Let $G=\mathrm{Isom}(X)$ be the isometry group of a geodesic, proper and Gromov-hyperbolic metric space $X$.\\

The isometric action of $G$ on $X$ extends to a convergence action on the Gromov compactification $\overline{X}$ by Theorem \ref{Tukia}. Let $\Gamma$ be a discrete subgroup of $G$ and $x\in X$ be a basepoint. The set $\Lambda(\Gamma)= \overline{\Gamma x} \cap \partial X$ is called the \textit{limit set} of $\Gamma$. It does not depend on the choice of the basepoint $x\in X$ and is the smallest closed and $\Gamma$-invariant subset of $\partial X$. An element $z\in \partial X$ is in $\Lambda(\Gamma)$ if and only if there is a convergence sequence $(\gamma_n)\in \Gamma^\mathbb{N}$ with base $(z,z')$ for some $z'\in \partial X$. The limit set contains every fixed point of hyperbolic and parabolic elements of $\Gamma$.\\

An element $z \in \Lambda(\Gamma)$ is said to be a \textit{conical point} if there are $z',z'' \in \Lambda(\Gamma)$ with $z'\neq z''$, and a convergence sequence $(\gamma_n)\in \Gamma^\mathbb{N}$ with base $(z,z')$ satisfying $\gamma_n(z)\rightarrow z''$. 

Every fixed point of a hyperbolic isometry is conical and no parabolic fixed point can be conical (see e.g. \cite[Theorem 3A]{Tukia2}). The fixed point $p \in \Lambda(\Gamma)$ of a parabolic element of $\Gamma$ is called \textit{bounded} if the stabilizer $\mathrm{Stab}_\Gamma(p)$ acts cocompactly on $\Lambda(\Gamma)\setminus \{p\}$.\\

We say that a discrete subgroup $\Gamma$ of $G$ is \textit{geometrically finite} if every point of $\Lambda(\Gamma)$ is either conical or a bounded parabolic fixed point. We say that a representation in $G$ is \textit{geometrically finite} if its image is a geometrically finite subgroup of $G$.

\begin{rem}
The original definition of geometrically finite subgroups for $G=\mathrm{PSL}_2(\mathbb{R})$ or $G=\mathrm{PSL}_2(\mathbb{C})$ was that one (or equivalently every) Dirichlet fundamental polyhedron has finitely many sides. Since then, several equivalent definitions have been established, including the one presented above. However, the situation is quite different for other isometry groups such as $\mathrm{Isom}(\mathbb{H}_{\mathbb{R}}^4)$ or $\mathrm{Isom}(\h)$ (see \cite[Theorem 4.5]{Apanasov} and \cite{Parker2}). In the more general setting where $G$ is the isometry group of a pinched Hadamard manifold (i.e. a complete and simply connected manifold with sectional curvature bounded between two negative constants), Bowditch showed in \cite{Bow1} that all the standard definitions of geometric finiteness remain equivalent, except the original one.
\end{rem}

\subsection{Boundary maps}

Let $\Gamma$ be a hyperbolic group and $G=\mathrm{Isom}(X)$ be the isometry group of a geodesic, proper and Gromov-hyperbolic metric space $X$.\\

If $\rho\in \mathcal{R}(\Gamma,G)$ is a discrete and faithful representation, then $\Gamma$ acts naturally on two compact spaces: $\partial \Gamma$, the Gromov boundary of its Cayley graph, and $\Lambda\bigl(\rho(\Gamma)\bigl)$, the limit set of $\rho(\Gamma)$. 

When $\rho$ is also geometrically finite, Gerasimov and Potyagailo generalized a theorem due to Floyd in \cite{Floyd}, to show that there is a nice equivariant map between them.

\begin{thm}[{\cite[Proposition 3.4.6 and its corollary]{Gerasimov} and \cite[Corollary 2.8]{Ger-Pot}}] \label{b-map}
\mbox{Let $\rho\in \mathcal{R}(\Gamma,G)$} be a discrete, faithful and geometrically finite representation. Suppose that the stabilizer in $\rho(\Gamma)$ of every bounded parabolic fixed point of $\Lambda\bigl((\rho(\Gamma)\bigl)$ is virtually cyclic. Then there is a surjective, continuous and $\rho$-equivariant map $\psi: \partial \Gamma \rightarrow \Lambda\bigl(\rho(\Gamma)\bigl)$ which is 2-to-1 onto bounded parabolic fixed points and injective elsewhere.
\end{thm}

\section{Horoballs and geometrical finiteness}

Let $G=\mathrm{Isom}(X)$ be the isometry group of a geodesic, proper and Gromov-hyperbolic metric space $X$.

\subsection{Busemann functions and horoballs}

For $z\in X$, the \textit{Busemann function} $\beta_z: X\times X \rightarrow \mathbb{R}$ is defined by:

$$\beta_z(x,y)=d(x,z)-d(y,z).$$

We also define Busemann functions for $z\in \partial X$ by setting, for all $x,y \in X$: 

$$\beta_z(x,y)= \sup \liminf_{n\rightarrow \infty} \beta_{z_n}(x,y)$$

where the supremum is taken over all sequences $(z_n)\in X^\mathbb{N}$ whose equivalence class is $z$.\\

Busemann functions are invariant under isometries. For each $z\in \overline{X}$ and $g\in G$, we have:

$$\beta_{g(z)}\bigl(g(.),g(.)\bigl)=\beta_z(.,.).$$

The following properties are classical:

\begin{prop}[{\cite[Chapitre 8, Section 1]{GdlH}}]\label{Bus}
There exists a constant $K(\delta)$, depending only on the hyperbolicity constant $\delta$ of $X$, such that for all $x,y,w,w' \in X$ and $z\in \partial X$:

\begin{enumerate}[label=\arabic*)]
\item $ \vert \beta_z(x,y) + \beta_z(y,x) \vert \leqslant K(\delta)$
\item $ \vert \beta_z(x,y) - \beta_z(x,w) -\beta_z(w,y) \vert \leqslant K(\delta)$
\item $ \vert \beta_z(x,y) - \beta_z(w,w') \vert \leqslant d(x,w) + d(y,w') + K(\delta)$.
\end{enumerate}

Moreover, if $(z_n)_{n\in\mathbb{N}}$ is a sequence of elements of $X$ and $o$ is a basepoint in $X$:

$$z_n \rightarrow z \quad \iff \quad \beta_z(z_n,o) \rightarrow -\infty$$
\end{prop}

We deduce the following two lemmas. 

\begin{lem}\label{Buse}
Let $o\in X$ be a basepoint, and $p\in \partial X$ be a point on the boundary. For every parabolic isometry $g\in G$ fixing $p$, we have:

$$\vert \beta_p(g(o),o) \vert \leqslant K(\delta)$$
\end{lem}

\begin{proof}
Let $g\in G$ be an isometry fixing $p$. The second item of Proposition \ref{Bus} together with the invariance of Busemann function by isometries implies that $\bigl(\beta_p(g^n(o),o)+K(\delta)\bigl)_{n\in\mathbb{N}}$ is a subadditive sequence and $\bigl(\beta_p(g^n(o),o)-K(\delta)\bigl)_{n\in\mathbb{N}}$ is a superadditive sequence. Hence 
$$B(g)= \inf_n \frac{\beta_p(g^n(o),o)+K(\delta)}{n}=\sup_n \frac{\beta_p(g^n(o),o)-K(\delta)}{n}= \lim _n \frac{\beta_p(g^n(o),o)}{n}$$

is well-defined and does not depend on the basepoint $o\in X$. As a consequence, we have: 
$$\vert \beta_p(g(o),o)-B(g) \vert \leqslant K(\delta)$$

It remains to show that $B(g)=0$ for any parabolic isometry $g$ fixing $p$. Using the third item of Proposition \ref{Bus}, we have for every $n\in \mathbb{N}$: 
$$\left\vert \beta_p(g^n(o),o)-\beta_p\bigl(o,g(o)\bigl) \right\vert \leqslant d\bigl(g^n(o),o\bigl)+d\bigl(o,g(o)\bigl)+K(\delta).$$

We deduce that $-N(g) \leqslant B(g)\leqslant N(g)$ where $N(g)=\lim_n \frac{d(o,g^n(o))}{n}$ is the stable norm of $g$. Hence $B(g)\neq 0$ implies that $N(g) > 0$, i.e., $g$ is hyperbolic.
\end{proof}

\begin{lem}\label{Busee}
Let $p\in \partial X$ be a point on the boundary and $(x_n)_{n\in \mathbb{N}}$ be a sequence of elements of $X$ converging to $p$. For every sequence $(g_n)_{n\in\mathbb{N}}$ of parabolic isometries of $G$ fixing $p$, the sequence $\bigl(g_n(x_n)\bigl)_{n\in\mathbb{N}}$ converges to $p$.
\end{lem}

\begin{proof}
Let $(x_n)_{n\in \mathbb{N}}$ be a sequence of elements of $X$ converging to $p$ and $(g_n)_{n\in\mathbb{N}}$ be a sequence of parabolic isometries of $G$ fixing $p$. Using the second property of Proposition \ref{Bus}, the invariance of Busemann function by isometries and Lemma \ref{Buse}, we get:

\begin{equation*} 
\begin{split}
\beta_p(g_n(x_n),o) & \leqslant \beta_p\bigl(g_n(x_n),g_n(o)\bigl)+\beta_p(g_n(o),o) +K(\delta)\\
& \leqslant \beta_p(x_n,o)+\beta_p(g_n(o),o) +K(\delta) \\
& \leqslant \beta_p(x_n,o) + 2K(\delta).
\end{split}
\end{equation*}

We deduce that $\beta_p(g_n(x_n),o) \rightarrow -\infty$ hence $\bigl(g_n(x_n)\bigl)_{n\in\mathbb{N}}$ converges to $p$ by Proposition \ref{Bus}.
\end{proof}

A horoball is a sublevel set of a Busemann function. If $z\in \partial X$, $o\in X$ is a basepoint and $r>0$, the \textit{horoball centered at $z$ of radius $r$} is 

$$H_z(r) = \{x\in X ~|~ \beta_z(x,o)\leqslant \ln(r) \}.$$

We will need the following lemma. 

\begin{lem}[{\cite[Corollary 2.4]{bray2021global}}]\label{qgh}
Let $z\in \partial X$, $o\in X$ be a basepoint, $r>0$ and $H_z(r)$ be a horoball centered at $z$ of radius $r$ which does not contain $o$.\\
There exists a constant $K'(\delta)$, depending only on the hyperbolicity constant $\delta$ of $X$, such that:

$$\left\vert \ln(r) + d\bigl(o,H_z(r)\bigl) \right\vert < K'(\delta).$$
\end{lem}

\subsection{Invariant family of horoballs and the cuspidal part}

Let $\Gamma$ be a discrete subgroup of $G$. Since $\Gamma$ acts isometrically on $X$ and preserves $\Lambda(\Gamma)$, $\Gamma$ also preserves the union of all biinfinite geodesics in $X$ which have both endpoints in $\Lambda(\Gamma)$. We denote by $C_\Gamma$ this $\Gamma$-invariant subset of $X$. Observe that $\Lambda(\Gamma)=\overline{C_\Gamma}\cap \partial X$.\\

Denote by $\mathcal{P}$ the set of parabolic fixed points of $\Gamma$. Note that $\mathcal{P}$ is a $\Gamma$-invariant subset of $\Lambda(\Gamma)$. An \textit{invariant family of horoballs} is a collection $\{H_p\}_{p\in \mathcal{P}}$ of pairwise disjoint horoballs $H_p$ centered at $p$, such that $ H_{\gamma(p)}= \gamma(H_p)$ for all $\gamma \in \Gamma$ and $p\in \mathcal{P}$. If, more generally, there exists a constant $C\geqslant 0$ such that the Hausdorff distance between $H_{\gamma(p)}$ and $\gamma(H_p)$ is less than $C$ for all $\gamma \in \Gamma$ and $p\in \mathcal{P}$, then we call $\{H_p\}_{p\in \mathcal{P}}$ a \textit{$C$-quasi-invariant family of horoballs}. For $R>0$, we say that such a family is \textit{$R$-separated} if $d(H_p,H_q)\geqslant R$ for all $p\neq q\in \mathcal{P}$.\\

Given an $R$-separated, $C$-quasi-invariant family $\mathcal{Q}=\{H_p\}_{p\in \mathcal{P}}$ of horoballs , the corresponding \textit{$\mathcal{Q}$-non-cuspidal part} for the action of $\Gamma$ on $X$ is the set:

$$X_{nc}^\mathcal{Q} = C_\Gamma \setminus \bigcup_{p\in \mathcal{P}} H_p, $$

and the \textit{$\mathcal{Q}$-cuspidal part} for the action of $\Gamma$ on $X$ is its complement in $C_\Gamma$:

$$X_{c}^\mathcal{Q} = \bigcup_{p\in \mathcal{P}} H_p \cap C_\Gamma.$$

If $\Gamma$ is a geometrically finite subgroup, Bray and Tiozzo adapted in \cite[Proposition 3.3]{bray2021global} a result due to Bowditch in \cite[Lemma 6.13]{Bow3}, and proved:

\begin{prop}\label{horo}
Let $\Gamma$ be a geometrically finite subgroup of $G$ and $\mathcal{P}\subset \partial X$ denote the set of all bounded parabolic fixed points of $\Gamma$. Then $\mathcal{P} / \Gamma$ is finite. Moreover, there are $C\geqslant 0$ and $R_0\geqslant 0$ such that for all $R\geqslant R_0$, there exists an $R$-separated family of $C$-quasi-invariant horoballs $\mathcal{Q}$ with $X_{nc}^\mathcal{Q}/\Gamma$ is compact.
\end{prop}

\begin{proof}
Bowditch, in \cite[Lemma 6.13]{Bow3}, uses a different definition of horoballs, defining them as level sets of horofunctions rather than Busemann functions. Bowditch’s horofunctions have the advantage of being independent of a basepoint. We call their level sets \textit{generalized horoballs}. In the proof of \cite[Proposition 3.3]{bray2021global}, Bray and Tiozzo noticed that there exists a constant $C(\delta)$, depending only on the hyperbolicity constant of $X$, such that every generalized horoball is within Hausdorff distance $C(\delta)$ of a horoball. The result now follows from \cite[Lemma 6.13]{Bow3}.
\end{proof}

\begin{rem}
An illustrative example is the following. Let $\Sigma$ be a hyperbolic surface with one cusp and $\Pi: \mathbb{H}^2 \rightarrow \Sigma $ denote the natural projection. If $\Omega$ is a neighbourhood of the cusp, then the preimage $\Pi^{-1}(\Omega)$ is contained in an invariant family of horoballs of $\mathbb{H}^2$ and $\pi_1(\Sigma)$ acts cocompactly on its complement (see Figure \ref{horo1}).\\
Moreover, if $(\Omega_n)_{n\in \mathbb{N}}$ is a sequence of neighbourhoods of the cusp that is decreasing (with respect to inclusion), then there exists an increasing sequence $(R_n)_{n\in\mathbb{N}}$ of real numbers such that $\Pi^{-1}(\Omega_n)$ is contained in a $R_n$-separated family of invariant horoballs for all $n\in \mathbb{N}$.
\end{rem}

\begin{figure}[!h]
\centering
\includegraphics[width=0.4\textwidth]{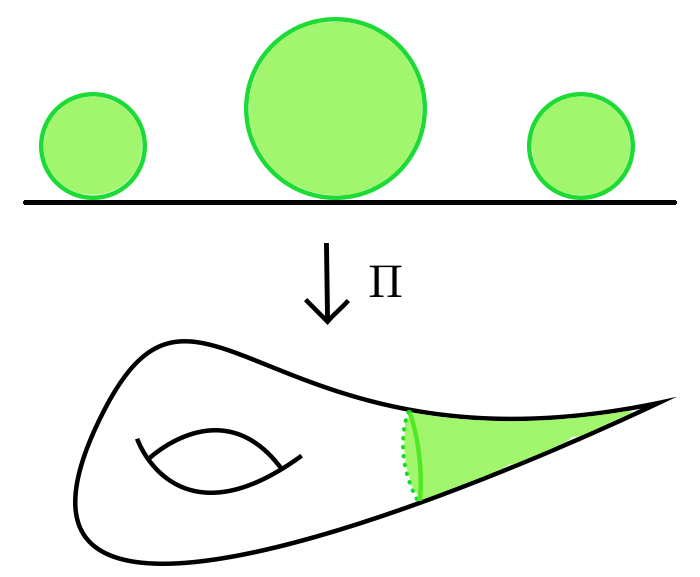}
\caption{An invariant family of horoballs of $\mathbb{H}^2$.}
\label{horo1}
\end{figure}

\begin{lem}\label{hor}
Let $C\geqslant 0$ and $(R_n)_{n\in \mathbb{N}}$ be a sequence of real numbers such that $R_n \rightarrow +\infty$. If $(\mathcal{Q}_n)_{n\in \mathbb{N}}=(\{H^n_p\}_{p\in \mathcal{P}})_{n\in \mathbb{N}}$ is a sequence of $R_n$-separated, $C$-quasi-invariant families of horoballs, then for all $p\in \mathcal{P}$, the radius of $H^n_p$ tends to 0 as $n$ tends to $\infty$.
\end{lem}

\begin{proof}
Fix a basepoint $o\in X$, let $p\in \mathcal{P}$ be a parabolic fixed point and let $x_n$ be a closest point projection of $o$ on $H_p^n$, i.e., for all $y_n\in H_p^n$, $d(o,x_n)\leqslant d(o,y_n)$.

Let $\gamma$ be an element of $\Gamma$, we have for all $n\in \mathbb{N}$:
\begin{equation*} 
\begin{split}
d\bigl(x_n,\gamma(x_n)\bigl) &\leqslant d(x_n,o) + d\bigl(o,\gamma(o)\bigl) + d\bigl(\gamma(o),\gamma(x_n)\bigl) \\
 & \leqslant 2d(x_n,o) + d\bigl(o,\gamma(o)\bigl).
\end{split}
\end{equation*}

Therefore if $\gamma$ does not fix $p$, then:
\begin{equation*} 
\begin{split}
d(o,x_n) &\geqslant \frac{1}{2} \Bigl(d\bigl(x_n,\gamma(x_n)\bigl)- d\bigl(o,\gamma(o)\bigl)\Bigl) \\
 & \geqslant \frac{1}{2} \Bigl(R_n-C - d\bigl(o,\gamma(o)\bigl)\Bigl).
\end{split}
\end{equation*}

We deduce that $\lim_n d(o,x_n) = \infty$ hence the radius of $H^n_p$ tends to 0 when $n$ tends to $\infty$ according to Lemma \ref{qgh}.
\end{proof}

\section{$A$-stable representations of hyperbolic groups} \label{A-sta}

In this section, we define $A$-stable representations of hyperbolic groups. These representations generalize primitive-stable representations of free groups in $\mathrm{PSL}_2(\mathbb{C})$ as defined by Minsky in \cite{minsky2013dynamics}.\\

Let $\Gamma$ be a hyperbolic group with a finite generating set $S$. Let $G=\mathrm{Isom}(X)$ be the isometry group of a geodesic and Gromov-hyperbolic metric space $X$ and fix a basepoint $x\in X$.

\subsection{Definition}

For an element $\gamma\in \Gamma$ of infinite order, we denote by $\gamma_-$ and $\gamma_+$ its fixed points in $\partial \Gamma$ for the extended action of $\gamma$ on $\overline{C_S(\Gamma)}$. We let $L_S(\gamma)$ be the set of geodesics in $C_S(\Gamma)$ joining $\gamma_-$ to $\gamma_+$. For a subset $A\subset \Gamma$, let $A_\infty$ be the subset of $A$ consisting of infinite order elements and
$$L_S(A)=\bigcup_{a\in A_\infty} L_S(a).$$

\begin{defi} 
Let $A \subset \Gamma$ be a subset of $\Gamma$.\\ 
A representation $\rho\in \mathcal{R}(\Gamma,G)$ is \textit{$A$-stable} if there exist constants $K\geqslant 1$ and $C\geqslant 0$ such that for every $l \in L_S(A)$, $\Or(l)$ is a ($K,C$) quasi-geodesic in $X$.
\end{defi}

When we discuss quasi-geodesics in $\Or\bigl(C_S(\Gamma)\bigl)$, we assume that the parametrization is given by arc length in $C_S(\Gamma)$, composed by $\Or$.\\

If we wish to emphasize the constants, we say that a representation $\rho \in \mathcal{R}(\Gamma,G)$ is a $(K,C)$ $A$-stable representation. While the notion of $(K,C)$ $A$-stability depends on the choices of the basepoint $x \in X$ and of the finite generating set $S$, the notion of $A$-stability does not, and hence is well defined on $\mathcal{X}(\Gamma,G)$. We denote by $S_A(\Gamma,G)\subset \mathcal{X}(\Gamma,G)$ the set of conjugacy classes of $A$-stable representations.

\subsection{Dynamics of $\mathrm{Out}(\Gamma)$ on $S_A(\Gamma,G)$}

It can be shown that when $A\subset \Gamma$ is $\mathrm{Aut}(\Gamma)$-invariant, the set $S_A(\Gamma,G)$ of conjugacy classes of $A$-stable representations is $\mathrm{Out}(\Gamma)$-invariant.\\

We turn our attention on the dynamics of $\mathrm{Out}(\Gamma)$ on such subsets $S_A(\Gamma,G)$ but first we give some definitions. Let $H$ be a group acting continuously on a topological space $Y$ and $Y'\subset Y$ be an $H$-invariant subset of $Y$. Recall that $H$ acts \textit{properly discontinuously} on $Y'$ if for every compact set $R\subset Y'$, the set $\{h\in H ~|~ h(R)\cap R \ne \emptyset\}$ is finite.
We say that $Y'$ is a \textit{domain of discontinuity} if $Y'$ is open, $H$-invariant, and the action of $H$ on $Y'$ is properly discontinuous.

\begin{defi}\label{wt}
A subset $A\subset \Gamma$ is \textit{testable} if, for every sequence $(\varphi_n)_{n\in\mathbb{N}}$ of elements of $\mathrm{Aut}(\Gamma)$ which projects onto a sequence of distinct elements of $\mathrm{Out}(\Gamma)$, there is $a\in A$ such that: $$\limsup_{n\rightarrow \infty}~ l_S\bigl(\varphi_n(a)\bigl) = \infty$$
where $l_S\bigl(\varphi_n(a)\bigl)$ denotes the translation length of $\varphi_n(a)$ for its action on $C_S(\Gamma)$.
\end{defi}

Note that this property does not depend on the choice of the set $S$ of generators and that the translation length can be replaced by the stable norm by \cite[Chapter 10, Proposition 6.4]{coornaert2006geometrie}. The following proposition is proved in \cite[Theorem 4.10]{R}.

\begin{prop}\label{Apd}
Let $A\subset \Gamma$ be a testable and $\mathrm{Aut}(\Gamma)$-invariant subset of $\Gamma$. The subset $S_A(\Gamma,G)$ is a domain of discontinuity for the action of $\mathrm{Out}(\Gamma)$ on $\mathcal{X}(\Gamma,G)$. 
\end{prop}

Note that for any subset $A\subset \Gamma$, the set $S_A(\Gamma,G)$ contains $\mathrm{CC}(\Gamma,G)$ but that it can be challenging to determine whether this inclusion is strict. 

\subsection{Another characterization}

This section ends with another characterization of $A$-stable representations due to Lee in \cite[Lemma III.10]{lee2012dynamics} for discrete and faithful representations in $\mathrm{PSL}_2(\mathbb{C})$.

\begin{defi}\label{Acc}
We say that a representation $\rho\in \mathcal{R}(\Gamma,G)$ is \textit{$A$-quasi-convex} if:
\begin{itemize}
\item $\rho(a)$ is hyperbolic for every $a\in A_\infty$,
\item There exists a connected $\rho(\Gamma)$-invariant subset $\Omega$ of $X$, on which the action is cocompact and such that for all $a$ in $A_\infty$, every geodesic joining $\rho(a)_-$ and $\rho(a)_+$ is contained in $\Omega$.
\end{itemize}
\end{defi}

The next proposition shows the equivalence between the two previous notions.

\begin{prop}\label{Aqc}
A representation $\rho\in \mathcal{R}(\Gamma,G)$ is $A$-stable if and only if it is $A$-quasi-convex.
\end{prop}

\begin{proof}
If $\rho$ is $(K,C)$ $A$-stable, then the image under $\rho$ of every element $a\in A_\infty$ is hyperbolic. Let $\Or$ be an orbit map associated with $\rho$. The Morse lemma gives us a constant $D>0$, depending only on $K,C$ and the hyperbolicity constant of $X$, such that the Hausdorff distance between a geodesic in $X$ joining $\rho(a)_-$ and $\rho(a)_+$, and the image $\tau_{\rho,x}(l_a)$ of any geodesic $l_a$ in $C_S(\Gamma)$ joining $a^-$ and $a^+$, is at most $D$. Consequently, we can take $\Omega$ to be the $2D$-neighbourhood of $\Or\bigl(C_S(\Gamma)\bigl)$ which is a (path-)connected subset of $X$ on which $\rho(\Gamma)$ acts cocompactly.\\

Conversely, suppose that $\rho(a)$ is hyperbolic for all $a\in A_\infty$ and that there exists a connected, $\rho(\Gamma)$-invariant subset $\Omega\subset X$ on which the action is cocompact, and which contains every geodesic in $X$ joining the fixed points of elements of $A$. Without loss of generality, we may assume that $\Omega$ contains the image of $C_S(\Gamma)$ under an orbit map (if not, pick $x\in \Omega$ and replace $\Omega$ with $\Omega \cup \tau_{\rho,x}\bigl(C_S(\Gamma)\bigl)$ which is connected since it is the union of two connected subsets with non-empty intersection). Denote by $d_{\Omega}$ the intrinsic metric on $\Omega$ and note that $(\Omega,d_{\Omega})$ is a length space. According to the \v{S}varc-Milnor lemma, there exist $K\geqslant 1$ and $C\geqslant 0$ such that the metric space $(\Omega,d_{\Omega})$ is $(K,C)$ quasi-isometric to $(C_S(\Gamma),d_S)$. Therefore, $(\Omega,d_{\Omega})$ is Gromov-hyperbolic (see e.g. \cite[Theorem 3.20]{Vai}) and there exists a constant $D>0$, depending on $K,C$ and the hyperbolicity constant of $(\Omega,d_{\Omega})$, such that the image under $\Or$ of every geodesic $l_a$ in $L_S(A)$ is at Hausdorff distance at most $D$ from any geodesic $l$ in $X$ joining $\rho(a)^-$ and $\rho(a)^+$. Since $l$ is contained in $\Omega$ by assumption, it is a geodesic for both the intrinsic metric $d_{\Omega}$ and the ambient metric $d$. Hence, if $x,y$ lie on $\tau_{\rho,x}(l_a)$ and if $p_x$ and $p_y$ denote closest point projections (for $d_{\Omega}$) of $x$ and $y$ onto $l$, then 

$$d_{\Omega}(x,y)\leqslant d_{\Omega}(p_x,p_y)+2D = d(p_x,p_y)+2D \leqslant d(x,y) +4D.$$

Hence, $\tau_{\rho,x}(l_a)$ is a $(K,C+4D)$ quasi-geodesic in $(X,d)$.
\end{proof}

\section{Simple-stable representations}\label{V}

Let $\Gamma_g =\pi_1(\Sigma_g)$ be the fundamental group of a closed, connected and orientable surface $\Sigma_g$ of genus $g\geqslant 2$. Let $G=\mathrm{Isom}(X)$ be the isometry group of a Gromov-hyperbolic, geodesic and proper metric space $X$.\\

Elements of $\Gamma_g=\pi_1(\Sigma_g)$ correspond bijectively to free homotopy classes of closed curves on $\Sigma_g$. We say that a non-trivial element of $\Gamma_g$ is \textit{simple} if it corresponds to the free homotopy class of a closed curve on $\Sigma_g$ which does not self-intersect. We denote by $\mathcal{S}$ the set of simple elements of $\Gamma_g$.

\subsection{Domain of discontinuity}

We first show the following proposition: 

\begin{prop}\label{DD}
The set $S_{\mathcal{S}}(\Gamma_g, G)$ is a domain of discontinuity for the action of $\mathrm{Out}(\Gamma_g)$ on $\mathcal{X}(\Gamma_g,G)$.
\end{prop}

\begin{proof}
By Proposition \ref{Apd}, it is sufficient to show that $\mathcal{S}$ is a testable and $\mathrm{Aut}(\Gamma_g)$-invariant subset of $\Gamma_g$. The subset $\mathcal{S}$ is clearly invariant under inner automorphisms. Moreover, every outer automorphism of $\Gamma_g$ arises from the homotopy class of a homeomorphism of $\Sigma_g$ by Dehn-Nielsen-Baer theorem. Therefore, $\mathcal{S}$ is $\mathrm{Aut}(\Gamma_g)$-invariant. It remains to show that $\mathcal{S}$ is testable. Let
$$\Gamma_g=\langle a_1,b_1,...a_g,b_g ~\vert~ [a_1,b_1]...[a_g,b_g] \rangle $$
be the classical presentation of $\Gamma_g$. The set $S=\{a_1,b_1^{-1}, a_2, b_2^{-1},..., a_g,b_g^{-1}\}$ generates the group $\Gamma_g$, and one can check that every element in $S$ and every product of distinct elements of $S$ is simple. This verification can be done by drawing the curves on $\Sigma_g$ or by using the algorithm from \cite{Chillingworth}. It follows that $\mathcal{S}$ is testable by \cite[Theorem 5.12]{R}. Hence $S_{\mathcal{S}}(\Gamma_g, G)$ is a domain of discontinuity.
\end{proof}

\begin{rem}
Another argument to show that $\mathcal{S}$ is testable is the following. Consider two simple closed curves $\alpha$ and $\beta$ on $\Sigma_g$ such that $\Sigma_g\setminus (\alpha\cup\beta)$ is a finite union of $k$ disks. Using Alexander's trick, we see that the stabilizer of $\alpha \cup \beta$ in $\mathrm{MCG}(\Sigma_g)$ is finite, bounded by $k!$. Since the number of simple closed curves whose lengths are bounded by a given constant is finite, we obtain that $\{a,b\}$ is a testable subset of $\mathcal{S}$, where $a,b\in \mathcal{S}$ represent $\alpha$ and $\beta$, respectively.
\end{rem}

\subsection{Simple closed curves on closed surfaces}

To determine whether an element of $\Gamma_g$ is simple ,it is convenient to endow $\Sigma_g$ with a hyperbolic structure, i.e., fix a discrete and faithful representation $\rho_0\in \mathcal{R}\bigl(\Gamma_g,\mathrm{PSL}_2(\mathbb{R})\bigl)$. In that case, a non-trivial closed curve on $\Sigma_g \simeq \mathbb{H}^2 / \rho_0(\Gamma_g)$ has a unique geodesic representative in its free homotopy class. Such a class is simple if and only if its geodesic representative does not self-intersect (see \cite[Proposition 1.6]{farb2011primer}). Equivalently, this occurs if and only if the lifts in $\mathbb{H}^2$ of the geodesic representative do not intersect. Note that if $\gamma\in \Gamma_g$, then a lift of the geodesic representative on $\Sigma_g$ is given by the axis of $\rho_0(\gamma)$ for its action on $\mathbb{H}^2$ (that is the unique geodesic joining the attracting and repelling fixed points of $\rho_0(\gamma)$ in $\partial \mathbb{H}^2$).\\

Since every homotopy equivalence between closed surfaces is homotopic to a homeomorphism (see e.g. \cite[Theorem 8.9]{farb2011primer}), the property of being simple is independent of the chosen hyperbolic structure.

We deduce the following proposition

\begin{prop}\label{malin}
Let $\alpha\in \pi_1(\Sigma_g)$ be the free homotopy class of a closed curve on $\Sigma_g$. The following statements are equivalent.

\begin{itemize}
\item $\alpha$ is simple,
\item there exists a hyperbolic structure on $\Sigma_g$ for which the geodesic representative of $\alpha$ does not self-intersect,
\item for every hyperbolic structure on $\Sigma_g$, the geodesic representative of $\alpha$ does not self-intersect.
\end{itemize}
\end{prop}

Fix a hyperbolic structure on $\Sigma_g$ and write $\gamma^-$ and $\gamma^+$ the fixed points in $\partial \mathbb{H}^2$ of an element $\gamma\in \Gamma_g$ for its action on $\mathbb{H}^2$. We now establish the following proposition.

\begin{prop}\label{simple}
Let $h$ be a non-simple element of $\Gamma_g$ and $(\gamma_n)_{n\in \mathbb{N}}$ be a sequence of elements of $\Gamma_g$. Suppose that there exists $h^\pm \in \{h^-,h^+\}$ and $K$ a compact subset of $\partial \mathbb{H}^2 \setminus \{h^\pm\}$ such that $\gamma_n^+ \rightarrow h^\pm$ and $\gamma_n^- \in K$ for all $n\in \mathbb{N}$.\\
Then for sufficiently large $n$, $\gamma_n$ is a non-simple element of $\Gamma_g$.
\end{prop}

\begin{proof}
Denote by $L_h$ (resp. $L_n$) the geodesic in $\mathbb{H}^2$ joining $h^-$ and $h^+$ (resp. $\gamma_n^-$ and $\gamma_n^+$). We work in the upper half plane model where $\partial \mathbb{H}^2 \simeq \mathbb{R}\cup \{\infty\}$. Suppose that $h^-=0$, $h^+=\infty$, $\gamma_n^+ \rightarrow \infty$ and there exists a compact interval $K$ of $\mathbb{R}$, centered at $0$, such that $\gamma_n^- \in K$ for all $n\in \mathbb{N}$. 

Since $h$ is non-simple, there is $g_0 \in \Gamma_g$ such that $g_0(L_h) \cap L_h \neq \emptyset$. This implies that the endpoints $a$ and $b$ of $g_0(L_h)$ have opposite sign (see Figure \ref{Figurerab}). We only treat here the case when $\gamma_n^+ \rightarrow +\infty$, $a=g_0(0)<0$ and $b=g_0(\infty)>0$ but the other cases can be treated similarly.
Let $U_a$ be an open neighbourhood of $a$ which is relatively compact in $\mathbb{R}_-^*$ and $U_0$ be an open neighbourhood of $0$ such that $g_0(U_0) \subset U_a$. The isometry $h$ fixes $0$ and $\infty$ and the sequence $(h^n)_{n\in \mathbb{N}}$ is a convergence sequence with base $(0,\infty)$. Therefore there exists $k_0\in \mathbb{N}$ such that:
$$h^{-k_0}(K)\subset U_0 \quad, \quad h^{k_0}(U_a) \subset \mathbb{R}_- \setminus K \quad, \quad h^{k_0}(b)\in \mathbb{R}_+ \setminus K.$$

We deduce:
$$h^{k_0}g_0h^{-k_0}(K) \subset \mathbb{R}_- \setminus K \quad , \quad h^{k_0}g_0h^{-k_0}(\infty)=h^{k_0}(b) \in \mathbb{R}_+ \setminus K. $$ 

Since $\gamma_n^-\in K$, we have $h^{k_0}g_0h^{-k_0}(\gamma_n^-)<\gamma_n^-$ for all $n\in \mathbb{N}$. Moreover, $\gamma_n^+ \rightarrow +\infty$ implies $\lim_n h^{k_0}g_0h^{-k_0}(\gamma_n^+)=h^{k_0}g_0h^{-k_0}(\infty)=h^{k_0}(b)$. Hence for $n$ large enough, we have:
$$h^{k_0}g_0h^{-k_0}(\gamma_n^-)<\gamma_n^-<h^{k_0}g_0h^{-k_0}(\gamma_n^+)< \gamma_n^+.$$

Therefore, for sufficiently large $n$, the geodesics $L_n$ and $h^{k_0}g_0h^{-k_0}(L_n)$ intersect, which shows that $\gamma_n$ is non-simple. 
\end{proof}

\begin{figure}[!h]
\centering
\includegraphics[width=0.9\textwidth]{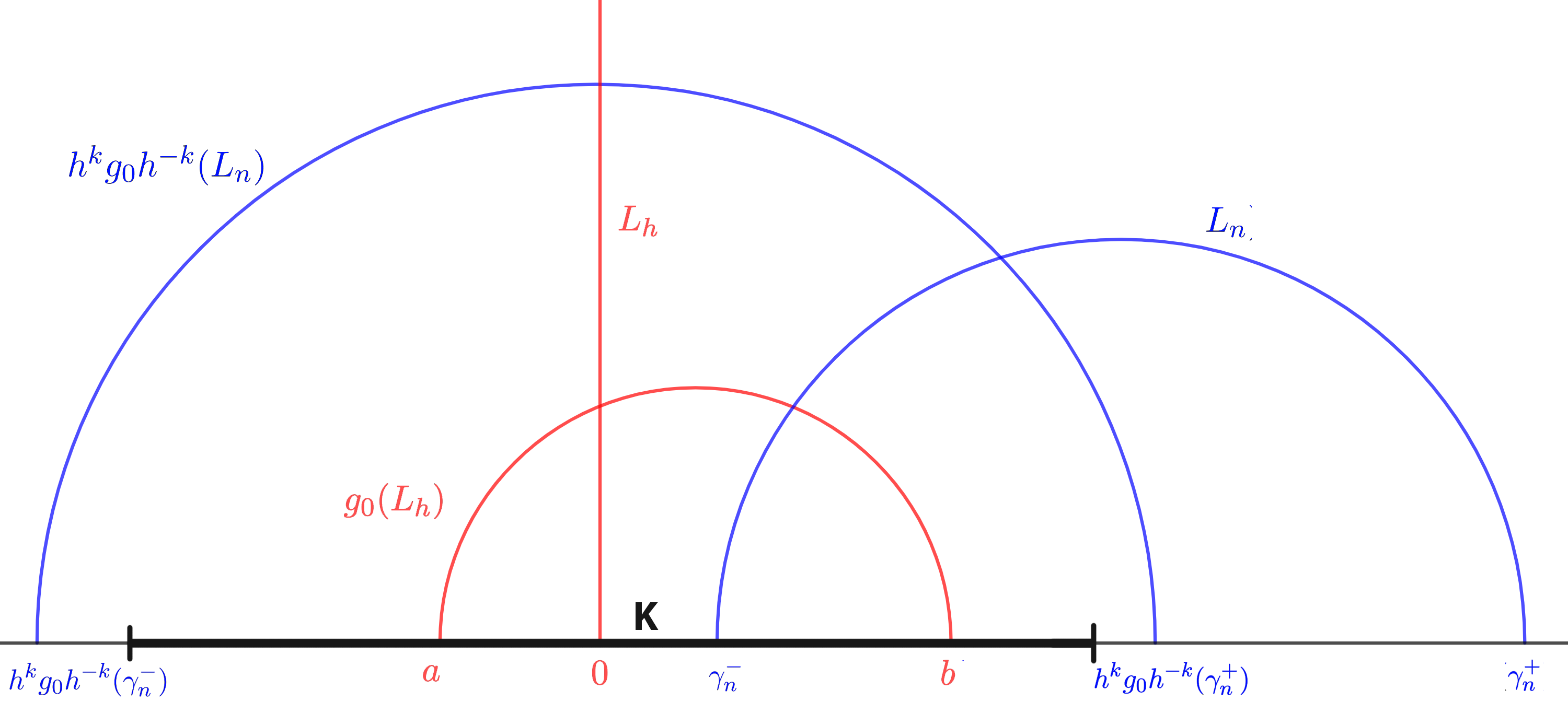}
\caption{The geodesics $L_h$, $g_0(L_h)$, $L_n$ and $h^{k_0}g_0h^{-k_0}(L_n)$.}
\label{Figurerab}
\end{figure}

\subsection{Geometrical finiteness and simple-stability}

We are now ready to prove Theorem \ref{Crit}. Although we could make a statement in our general framework, we feel like it is somehow more natural to strengthen our hypothesis on $G$. A \textit{pinched Hadamard manifold} is a complete, simply-connected manifold whose real sectional curvature lies between two negative real numbers. It is a geodesic, proper and Gromov-hyperbolic space (see e.g. \cite[Theorem II.4.1 and Proposition III.H.1.2]{bridson2013metric}). We have the following lemma:

\begin{lem}\label{cyclic}
Let $\Gamma$ be a hyperbolic group, $G$ be the isometry group of a pinched Hadamard manifold and let $\rho\in \mathcal{R}(\Gamma,G)$ be a discrete and faithful representation. If $p\in \Lambda\bigl(\rho(\Gamma)\bigl)$ is a parabolic fixed point, then the stabilizer $\mathrm{Stab}_{\rho(\Gamma)}(p)$ is virtually cyclic.
\end{lem}

\begin{proof}
Since $\rho(\Gamma)$ is discrete, $\mathrm{Stab}_{\rho(\Gamma)}(p)$ cannot contain both parabolic and hyperbolic elements (see \cite[Theorem 2G]{tukia1994convergence}). Hence, the limit set of $\mathrm{Stab}_{\rho(\Gamma)}(p)$ is a singleton and Bowditch proved in \cite{Bow0} that in that case, $\mathrm{Stab}_{\rho(\Gamma)}(p)$ is virtually nilpotent. Since a nilpotent group does not contain any free group of rank 2, the result follows from \cite[Corollary p.224]{ghys1989groupes}.
\end{proof}

\Crit*

\begin{proof}
The direct implication is already proved in Proposition \ref{Aqc} so we only focus on the reverse implication. Let $\rho\in \mathcal{R}(\Gamma_g,G)$ be a discrete, faithful and geometrically finite representation such that the image of every simple element of $\Gamma_g$ is hyperbolic.\\

If $\rho(\Gamma_g)$ is totally hyperbolic, then it follows from Bowditch's different characterizations of geometrically finite subgroups of $G$ in \cite{Bow1} that $\rho$ is convex-cocompact, hence simple-stable.\\
If $\rho(\Gamma_g)$ contains a parabolic element, we prove that $\rho$ is simple-stable by showing that $\rho$ is simple-quasi-convex (recall Definition \ref{Acc}).\\

Denote by $M_\rho = X/ \rho(\Gamma_g)$ the quotient space and $\Pi: X \rightarrow M_\rho $ the natural projection. Let $(\gamma_n)_{n\in \mathbb{N}}$ be a sequence of elements of $\Gamma_g$ such that $\rho(\gamma_n)$ is hyperbolic for all $n\in \mathbb{N}$. Denote by $L_n$ the geodesic in $X$ joining its fixed point $\rho(\gamma_n)^+$ and $\rho(\gamma_n)^-$ and suppose that for every compact subset $N$ of $M_\rho$, there is $n\in \mathbb{N}$ such that $\Pi(L_n) \not\subset N$. We need to show that for $n$ large enough, $\gamma_n$ is not simple.\\

We denote by $\mathcal{P}$ the set of parabolic fixed points in $\Lambda\bigl((\rho(\Gamma_g)\bigl)$. According to Proposition \ref{horo}, there is a sequence of real numbers $(R_n)_{n\in\mathbb{N}}$ with $ \lim_n R_n= +\infty$ and a sequence $(\mathcal{Q}_n)_{n\in \mathbb{N}}=(\{H^n_p\}_{p\in \mathcal{P}})_{n\in\mathbb{N}}$ of $R_n$-separated, $C$-quasi-invariant families of horoballs such that for all $n\in \mathbb{N}$, 
$$\Pi(L_n) \cap \Pi(X_c^{\mathcal{Q}_n}) \ne \emptyset$$ 
where $X_c^{\mathcal{Q}_n}$ denotes the $\mathcal{Q}_n$-cuspidal part for the action of $\Gamma$ on $X$.\\

Recall that there are only finitely many $\rho(\Gamma_g)$-orbits of parabolic fixed points in $\Lambda\bigl(\rho(\Gamma_g)\bigl)$ according to Proposition \ref{horo}. Hence, up to replacing $\gamma_n$ by one of its conjugates, we can suppose that there exists a bounded parabolic fixed point $p\in \mathcal{P}$, such that the geodesic $L_n$ intersects $H^n_p$ for all $n\in \mathbb{N}$.
Since the radius of $H^n_p$ tends to $0$ by Lemma \ref{hor}, there is $q_n\in L_n$ such that $\lim_n \beta_p(q_n,o) = -\infty$. Hence $q_n$ tends to $p$ by Proposition \ref{Bus}. This implies that, up to replacing some $\gamma_n$ by $\gamma_n^{-1}$, we can suppose that $\rho(\gamma_n)^+$ also tends to $p$ as $n$ tends to $+\infty$ (see Figure \ref{horo2}).\\

\begin{figure}[!h]
\centering
\includegraphics[width=0.4\textwidth]{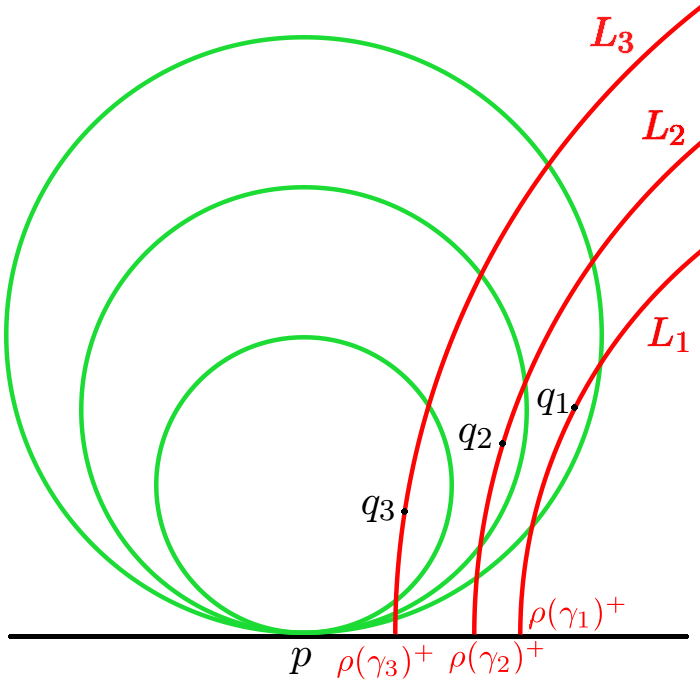}
\caption{The sequences $(q_n)_{n\in \mathbb{N}}$ and $\bigl(\rho(\gamma_n)^+\bigl)_{n\in \mathbb{N}}$.}
\label{horo2}
\end{figure}

\textbf{\underline{Claim}:} Up to replacing $\gamma_n$ by one of its conjugates, we can suppose that there is a compact subset $K$ of $\Lambda\bigl(\rho(\Gamma_g)\bigl)\setminus \{p\}$ such that $\rho(\gamma_n)^+$ tends to $p$, and for all $n\in \mathbb{N}$, $\rho(\gamma_n)^-$ is contained in $K$.

\begin{proof}
Since $p$ is a bounded parabolic fixed point, $J_p=\mathrm{Stab}_{\rho(\Gamma_g)}(p)$ acts cocompactly on $\Lambda\bigl(\rho(\Gamma_g)\bigl)\setminus \{p\}$. Because $\rho(\gamma_n)^-$ is the fixed point of a hyperbolic element, it must be a conical point, hence $\rho(\gamma_n)^-\neq p$ for all $n\in \mathbb{N}$. Therefore, there exist a compact subset $K$ of $\Lambda\bigl(\rho(\Gamma_g)\bigl)\setminus \{p\}$ and a sequence $(j_n)_{n\in\mathbb{N}}$ of elements of $J_p$ such that for all $n\in \mathbb{N}$, $j_n (\rho(\gamma_n)^-)\in K$. Lemma \ref{Busee} proves that $j_n(q_n)\rightarrow p$. Since $j_n(q_n)\in j_n(L_n)$ and $j_n(\rho(\gamma_n)^-)$ lies in a compact subset of $\Lambda\bigl(\rho(\Gamma_g)\bigl)\setminus \{p\}$, we must have $j_n(\rho(\gamma_n)^+) \rightarrow p$. We conclude by considering $j_n\gamma_n j_n^{-1}$ instead of $\gamma_n$.
\end{proof}

Since $G$ is the isometry group of a pinched Hadamard manifold, $J_p=\mathrm{Stab}_{\rho(\Gamma_g)}(p)$ is a cyclic subgroup of $\rho(\Gamma_g)$ by Lemma \ref{cyclic}. Moreover, $J_p$ is generated by the image of a non-simple element $h\in \Gamma_g$ by assumption. Denote by $h^-$ and $h^+$ the fixed points of $h$ for its action on $\partial \Gamma_g$.\\

Combining the claim with the existence of the continuous and $\rho$-equivariant map given by Theorem \ref{b-map}, we obtain:

\begin{itemize}
\item For every open neighbourhood $\mathcal{N}(h^-,h^+)$ of $\{h^+,h^-\}$, there is $N\in \mathbb{N}$ such that for all $n\geqslant N$, $\gamma_n^+ \in \mathcal{N}(h^-,h^+)$,
\item The sequence $(\gamma_n^-)_{n\in \mathbb{N}}$ lies in a compact subset of $\partial \Gamma_g \setminus \{h^-,h^+\}$.
\end{itemize} 

If we endow $\Sigma_g$ with a hyperbolic structure, i.e., fix a discrete and faithful representation $\rho_0\in \mathcal{R}\bigl(\Gamma_g, \mathrm{PSL}_2(\R) \bigl)$, then $\partial \Gamma_g$ and $\partial \mathbb{H}^2$ are $\rho_0$-equivariantly homeomorphic. Therefore, we can use Proposition \ref{simple} to deduce that that for $n$ large enough, $\gamma_n$ corresponds to a non-simple element of $\Gamma_g$. This concludes the proof that $\rho$ is simple-quasi-convex, hence simple-stable according to Proposition \ref{Aqc}.
\end{proof}

\section{Complex hyperbolic geometry}

In this section, we introduce the complex hyperbolic plane $\h$ and its geometry. The complex hyperbolic spaces $\mathbb{H}_\mathbb{C}^n$ are the natural analogs of the more widely studied real hyperbolic spaces $\mathbb{H}_\mathbb{R}^n$, but in the setting of complex geometry. The real hyperbolic spaces $\mathbb{H}_\mathbb{R}^n$ can be realized via a projective model as the space of negative lines in $\mathbb{R}^{n+1}$ with respect to a quadratic form of signature $(n,1)$. Similarly, by considering the space of negative lines in $\mathbb{C}^{n+1}$ with respect to a Hermitian form of signature $(n,1)$, one obtains a projective model of complex hyperbolic space $\mathbb{H}_\mathbb{C}^n$. The content presented in the first five subsections is fairly standard. We refer the reader to \cite{GoldmanComplex},\cite{ParkerNotes}, and \cite{Kapovich} for a more detailed introduction to complex hyperbolic geometry.

\subsection{Complex hyperbolic plane}

Let $\mathbb{C}^{2,1}$ denote the complex vector space of complex dimension $3$ equipped with a non-degenerate, indefinite Hermitian form $\langle \cdot,\cdot \rangle$ of signature $(2,1)$. This Hermitian form can be represented by a $3\times 3$ Hermitian matrix with $2$ positive eigenvalues and 1 negative eigenvalue. Two standard choices of such matrices are:
 
$$J_1=\begin{pmatrix}
1 & 0 & 0 \\
0 & 1 & 0 \\
0 & 0 & -1
\end{pmatrix} ~~~~\mathrm{  and  }~~~~
J_2=\begin{pmatrix}
0 & 0 & 1 \\
0 & 1 & 0 \\
1 & 0 & 0
\end{pmatrix}$$

If $z=\begin{pmatrix} z_1, z_2 , z_3 \end{pmatrix}^t$ and $w=\begin{pmatrix} w_1,w_2,w_3 \end{pmatrix}^t$ are vectors of $\mathbb{C}^{2,1}$, then the \textit{first Hermitian form} is defined to be:

$$\langle z,w \rangle_1 =\overline{w}^{t} J_1 z = z_1\overline{w_1} + z_2\overline{w_2} -z_3\overline{w_3}, $$

and the \textit{second Hermitian form} is defined to be: 

$$\langle z,w \rangle_2 = \overline{w}^{t} J_2 z = z_1\overline{w_3} + z_{2}\overline{w_2} +z_3\overline{w_1}. $$

We sometimes specify which of these two Hermitian forms we use and sometimes drop the subscript to say that we can use either of these.

Observe that $\langle z,z \rangle$ is real for all $z\in \mathbb{C}^{2,1}$, therefore, the following subsets form a partition of $\mathbb{C}^{2,1}$:
$$V_-= \{ z\in \mathbb{C}^{2,1} ~|~ \langle z,z \rangle < 0 \}$$
$$V_0= \{ z\in \mathbb{C}^{2,1} ~|~ \langle z,z \rangle = 0 \}$$
$$V_+= \{ z\in \mathbb{C}^{2,1} ~|~ \langle z,z \rangle > 0 \}$$

We say that $z\in \mathbb{C}^{2,1}$ is \textit{negative}, \textit{null} or \textit{positive} depending on whether $z$ belongs to $V_-$, $V_0$ or $V_+$. Since $\langle \lambda z,\lambda z \rangle= \lvert \lambda \rvert^2 \langle z,z \rangle $ for any non-zero complex number $\lambda$, the above definitions extend to the set of (complex) lines of $\mathbb{C}^{2,1}$ and also define a partition of $\mathbb{CP}^2$. Denote by $p:z \rightarrow \left[z\right]$ the natural projection of $\mathbb{C}^{2,1} \setminus \{0\}$ to $\mathbb{CP}^2$. \\

If $z\notin V_0$, the tangent space $T_{[z]}\mathbb{CP}^2$ identifies with $z^\perp$, the orthogonal complement for $\langle \cdot,\cdot \rangle$ of $\mathbb{C}z$ in $\mathbb{C}^{2,1}$. If $z\in V_-$, then the restriction of $\langle \cdot,\cdot \rangle$ to $z^\perp$ is a positive-definite Hermitian form. Hence, $\langle.,.\rangle$ endows $p(V_-)$ with a Hermitian metric, called the \textit{Bergman metric} and denoted by $h$.

\begin{defi}
We denote by $\mathbb{H}^2_{\mathbb{C}}$ the image $p(V_-)$ endowed with the Bergman metric $h$. We call it the \textit{projective model of the complex hyperbolic plane for $\langle.,.\rangle$}. Its boundary $\partial \mathbb{H}^2_{\mathbb{C}}$ is the image $p(V_0)$.
\end{defi}

Let $z=\begin{pmatrix} z_1, z_2 , z_3 \end{pmatrix}^t$ be a vector of $\mathbb{C}^{2,1}\setminus \{0\}$.\\
Note that $\langle z,z \rangle_1 \leqslant 0$, implies $z_3\neq 0$, hence we can consider the section $\{z_3=1\}$ to identify $\h$ (respectively $\partial \h$) with the open unit ball of $\mathbb{C}^2$ (resp. the unit sphere $S^3$ in $\mathbb{C}^2$). We call it the \textit{unit ball model} of the complex hyperbolic space.\\

Similarly, $\langle z,z \rangle_2 < 0$ implies $z_3\neq 0$. Therefore, we can consider the section $\{z_3=1\}$ to identify $\h$ with the subset of $\mathbb{C}^2$ defined by:
$$2\Re(z_1) + \vert z_2 \vert^2 <0$$
where $\Re$ denotes the real part. \\
We call it the \textit{Siegel model} of $\h$. Moreover, $\langle z,z \rangle_2 = 0$ implies $z_3\neq 0$ or $z_2=z_3=0$. The boundary $\partial \h$ therefore identifies with the one-point compactification of the paraboloid 
$$2\Re(z_1) + \vert z_2 \vert^2 =0.$$
Points on the paraboloid are called finite points and we write $\infty=\left[\begin{pmatrix} 1 , 0 , 0 \end{pmatrix}^t \right]$.\\

The Bergman metric $h$ on $\mathbb{H}^2_{\mathbb{C}}$ determines a Riemannian metric $g$, given by the real part of $h$:
$$g=\frac{1}{2}(h+\overline{h}).$$

It also defines a complex differential $2$-form $\omega$ of degree $(1,1)$, defined as minus the imaginary part of $h$:
$$\omega=\frac{i}{2}(h-\overline{h}).$$

The form $\omega$ is closed, hence equips $\mathbb{H}^2_{\mathbb{C}}$ with a Kähler structure. The distance function $d$ on $\mathbb{H}^2_{\mathbb{C}}$ induced by the Riemannian metric $g$ is called the \textit{Bergman distance}. If $z$ and $w$ are two negative vectors of $\mathbb{C}^{2,1}$, the Bergman distance between $[z]$ and $[w]$ is given by: 
$$\mathrm{cosh}^2 \left(\frac{d([z],[w])}{2}\right)=\frac{\langle z,w \rangle \langle w,z \rangle}{\langle z,z \rangle \langle w,w \rangle}.$$

\subsection{Isometry group}

We now describe the isometry group of $\h$. Since the Bergman distance can be expressed in terms of the Hermitian form $\langle \cdot,\cdot \rangle$, any linear transformation of $\mathbb{C}^{2,1}$ preserving it induces an isometry of $\mathbb{H}^2_{\mathbb{C}}$. A linear transformation represented by a matrix $A\in M_3(\mathbb{C})$ preserves the Hermitian form $\langle \cdot, \cdot \rangle$ if and only if $\overline{A}^{t}JA=J$. We denote by $\mathrm{U}(2,1)$ the set of such matrices and call it \textit{unitary with respect to} $\langle \cdot, \cdot \rangle$.
Note that the determinant of every matrix in $\mathrm{U}(2,1)$ has unit modulus and that any matrix in $\mathrm{U}(1)= \{e^{i\theta} I_3 ~|~ \theta\in [0,2\pi)\}$ acts trivially on the set of lines of $\mathbb{C}^{2,1}$ and hence on the complex hyperbolic plane. The projective unitary group $\mathrm{PU}(2,1):=\mathrm{U}(2,1)/\mathrm{U}(1)$ acts faithfully on $\mathbb{H}^2_{\mathbb{C}}$. Additionally, the complex conjugation map in $\mathbb{C}^{2,1}$ also induces an isometry of $\mathbb{H}^2_{\mathbb{C}}$. The following proposition shows that the full group of isometries of $\mathbb{H}^2_{\mathbb{C}}$ is generated by elements of $\mathrm{PU}(2,1)$ and the conjugation map.

\begin{prop}[{\cite[Theorem 3.5]{ParkerNotes}}]
Every isometry of $\h$ is either holomorphic or else anti-holomorphic. Moreover, each holomorphic isometry of $\h$ is given by an element of $\mathrm{PU}(2,1)$ and each anti-holomorphic isometry of $\h$ is given by complex conjugation followed by an element of $\mathrm{PU}(2,1)$.
\end{prop}

We sometimes consider instead of $\mathrm{PU}(2,1)$, the group $\mathrm{SU}(2,1)$ of unitary matrices with respect to $\langle \cdot, \cdot \rangle$ which have determinant 1. Indeed, $\mathrm{PU}(2,1)=\mathrm{SU}(2,1)/\{I_3, \omega I_3, \omega^2 I_3\}$ where $\omega=e^{\frac{2i\pi}{3}}$, hence $\mathrm{SU}(2,1)$ is a $3$-fold covering of $\mathrm{PU}(2,1)$. Both $\mathrm{PU}(2,1)$ and $\mathrm{SU}(2,1)$ have real dimension 8.\\

The complex hyperbolic plane is an example of a pinched Hadamard manifold whose (real) sectional curvature lies between $-1$ and $-1/4$ (see \cite[pp 72--79]{Goldman}).

As a consequence, the complex hyperbolic space is a geodesic and Gromov-hyperbolic space, hence every isometry must be hyperbolic, parabolic or elliptic. As in real hyperbolic geometry, the type of an isometry can be deduced from the trace of a matrix representing it. An elliptic isometry is said to be \textit{regular} if any of its lifts to $\mathrm{SU}(2,1)$ has three distinct eigenvalues. Any other elliptic isometry is called a \textit{complex reflection}. An element of $\mathrm{SU}(2,1)$ whose eigenvalues are all $1$ is called \textit{unipotent}. With the exception of the identity element, all unipotent elements are parabolic. Goldman proved the following theorem.

\begin{thm}[{\cite[Chapter 6]{GoldmanComplex}}]\label{trace}
Let $A\in \mathrm{PU}(2,1)$ be a holomorphic isometry of $\h$ and $\textbf{A}$ be a lift of $A$ to $\mathrm{SU}(2,1)$. Let $f:\mathbb{C} \rightarrow \mathbb{R}$ be the function defined by $f(z)= \lvert z\rvert ^4 - 8 \Re(z^3)+18\lvert z \rvert ^2 - 27$. 
\begin{itemize}
\item $f\bigl(\mathrm{tr}(\textbf{A})\bigl) > 0$ if and only if $A$ is hyperbolic.
\item $f\bigl(\mathrm{tr}(\textbf{A})\bigl) < 0$ if and only if $A$ is regular elliptic.
\item $f\bigl(\mathrm{tr}(\textbf{A})\bigl) = 0$ if and only if $A$ is parabolic or a complex reflection.
\item $\frac{1}{3}\mathrm{tr}(\textbf{A})$ is a cube root of unity if and only if $A$ is unipotent.
\end{itemize}
\end{thm}

The level set $f^{-1}(\{0\})$ is the deltoid $\Delta=\left\{2e^{i\theta} + e^{-2i\theta} ~|~ \theta\in [0,2\pi] \right \}$ (see \cite[End of Chapter 6]{GoldmanComplex}). It is represented in Figure \ref{Figure1}.

\begin{figure}[h]
\centering
\includegraphics[width=0.4\textwidth]{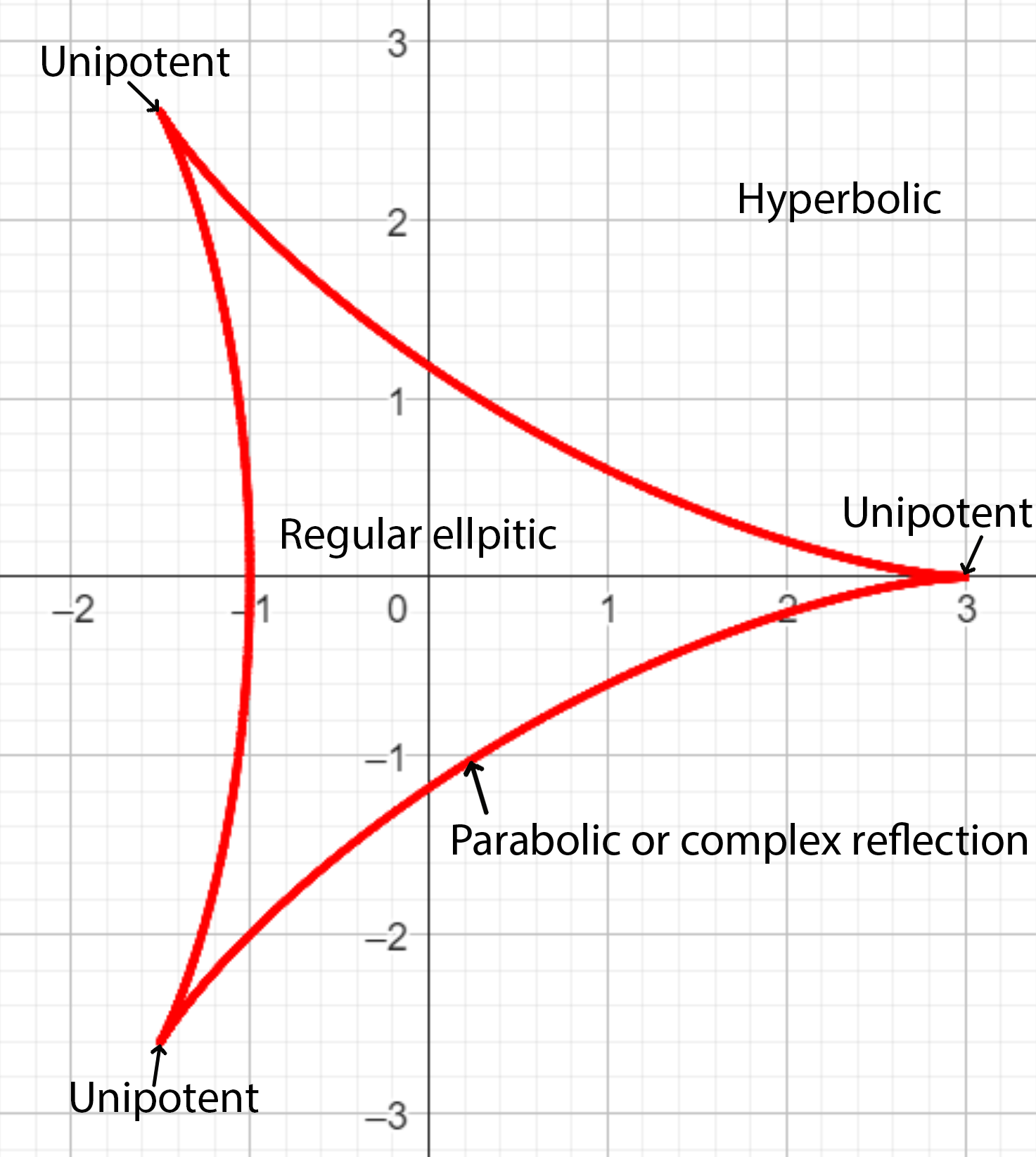}
\caption{The different values of $\mathrm{tr}(\textbf{A})\in \mathbb{C}$ with the deltoid $\Delta$.}
\label{Figure1}
\end{figure}

In contrast to $\mathrm{PO}_0(3,1)\simeq \mathrm{PSL}_2(\mathbb{C})\simeq \mathrm{Isom}^+(\mathbb{H}^3)$, the subset of parabolic elements in $\mathrm{PU}(2,1)$ has real codimension 1 (instead of 2) and there exist parabolic isometries which are not unipotent. Moreover, the subset of regular elliptic elements in $\mathrm{PU}(2,1)$ is an open subset whereas the subset of every elliptic elements in $\mathrm{PO}_0(3,1)$ has empty interior.

\subsection{Totally geodesic subspaces}

Unlike real hyperbolic spaces, $\h$ admits no (real) codimension-1 totally geodesic subspaces. However, there are two kinds of maximal totally geodesic subspaces.\\

Denote by $p: \mathbb{C}^{2,1}\setminus \{0\} \rightarrow \mathbb{CP}^2$ the natural projection map. A \textit{projective line} of $\mathbb{CP}^2$ is the image under $p$ of a complex $2$-dimensional linear subspace of $\mathbb{C}^{2,1}$. A \textit{totally real Lagrangian plane} of $\mathbb{CP}^2$ is the image under $p$ of a real $4$-dimensional linear subspace $V\subset \mathbb{C}^{2,1}$ such that $\langle v,w \rangle \in \mathbb{R}$ for all $v,w \in V$. A \textit{complex line} is the intersection of a projective line of $\mathbb{CP}^2$ with $\mathbb{H}^2_{\mathbb{C}}$ whenever non-empty. A \textit{real plane} is the intersection of a totally real Lagrangian plane of $\mathbb{CP}^2$ with $\h$ whenever non-empty.\\

The group $\mathrm{PU}(2,1)$ acts transitively on the set of complex lines and on the set of real planes. In the projective model of $\h$ for $\langle.,.\rangle_1$, every complex line (resp. real plane) is the image of $L_0$ (resp. $R_0$) under an element of $\mathrm{PU}(2,1)$ where:
$$L_0=\left\{\left[(z,0,1)^t\right] ~ \vert ~ z\in \mathbb{C}~,~ \vert z \vert < 1 \right\} \mathrm{~~~~~~and~~~~~~} R_0=\left\{ \left[(x_1, x_2,1)^t\right] ~\vert~ x_1, x_2 \in \mathbb{R} ~,~ x_1^2 + x_2^2 < 1 \right\}.$$

Both $L_0$ and $R_0$ are diffeomorphic to open discs. The restriction of the Bergman metric to $L_0$ is the Poincaré metric of constant curvature $-1$ whereas the restriction of the Bergman metric to $R_0$ is the Klein-Beltrami metric of constant curvature $-\frac{1}{4}$ (see \cite[Section 5]{ParkerNotes}).\\

Let $u,v \in \mathbb{C}^{2,1}$ be negative vectors such that $p(u)$ and $p(v)$ are distinct elements of $\h$. There is a unique complex line containing $p(u)$ and $p(v)$, namely, the intersection of $\h$ with the image under $p$ of the linear subspace spanned by $u$ and $v$.
If $L_1$ and $L_2$ are two distinct complex lines, we say that $L_1$ and $L_2$ are:

\begin{itemize}
\item \textit{concurrent} if they intersect in a single point of $\mathbb{H}^2_{\mathbb{C}}$.
\item \textit{asymptotic} if their closures intersect in a single point of $ \partial \mathbb{H}^2_{\mathbb{C}}$.
\item \textit{ultraparallel} if their closures are disjoint.
\end{itemize}

If $L_1$ and $L_2$ are concurrent complex lines, we define the \textit{angle} $\angle (L_1,L_2)$ between $L_1$ and $L_2$ as the smallest Riemannian angle formed by any geodesics $\gamma_1 \subset L_1$ and $\gamma_2\subset L_2$ passing through $L_1\cap L_2$.

\subsection{Heisenberg group and Cygan metric}

We investigate the structure of the boundary of $\h$ by looking at the action of unipotent isometries preserving $\langle .,. \rangle_2$ and fixing $\infty$. Every unipotent isometry fixing $\infty$ has the form:
$$T(\zeta,v) = \begin{pmatrix} 1 & -\sqrt{2}\overline{\zeta} & -\vert \zeta \vert ^2 + iv \\ 0 & 1 & \sqrt{2}\zeta \\ 0 & 0 & 1 \end{pmatrix},~~ \mathrm{where}~~ \zeta\in \mathbb{C},v\in\mathbb{R}$$

A direct computation shows that for all $\zeta,\xi\in \C$ and $v,t\in \R$ :
$$T(\zeta,v)^{-1}=T(-\zeta,-v) ~~~~ \mathrm{and} ~~~~ T(\zeta,v)T(\xi,t)=T\bigl(\zeta + \xi, v+t + 2\Im(\zeta\overline{\xi})\bigl)$$ 
where $\Im$ denotes the imaginary part.

Therefore, the group of unipotent isometries of $\mathrm{PU}(2,1)$ fixing $\infty$ is isomorphic to the \textit{Heisenberg group} $\mathcal{H}$, namely the group $\C\times \R$ endowed with the group law: 
$$(\zeta,v)\ast(\xi,t)=\bigl(\zeta + \xi, v+t + 2\Im(\zeta\overline{\xi})\bigl)$$

A finite point in the boundary of $\h$ is represented by the projectivization of a vector \mbox{$(z_1,z_2,1)^t\in \mathbb{C}^{2,1}$} such that $2\Re(z_1) + \vert z_2 \vert^2 =0$.

Writing $\zeta=z_2/ \sqrt{2}$ , we must have $z_1=-\vert \zeta \vert ^2 + iv$ for some $v\in \mathbb{R}$, therefore $(z_1,z_2,1)^t$ is the image of $(0,0,1)^t$ by $ T(\zeta,v)$.

Since every non trivial unipotent isometry is parabolic, the group of unipotent isometries fixing $\infty$ acts simply transitively on $\partial \h \setminus\{\infty\}$. Hence $\partial \h$ can be identified with the one-point compactification of $\mathcal{H}$.\\

The \textit{Heisenberg norm} of $(\zeta,v)\in \mathcal{H}$ is given by:
$$\vert (\zeta,v) \vert = \left\vert \vert \zeta \vert^2-iv \right\vert^{1/2}.$$

This norm induces a metric on $\mathcal{H}$, called the \textit{Cygan metric}, defined by:
$$\nu\bigl((\zeta,v),(\xi,t)\bigl)= \left\vert (\zeta,v)^{-1} \ast (\xi,t) \right\vert$$

Every unipotent isometry of $\h$ fixing $\infty$ preserves the Cygan metric on $\mathcal{H}$.

\begin{rem}
In a similar way, working with the upper half space model of $\mathbb{H}^n_{\R}$, the group of unipotent isometries fixing $\infty$ is isomorphic to $\R^{n-1}$. It acts simply transitively on $\partial \mathbb{H}^n_{\R} \setminus \{\infty\}$ and preserves the Euclidean metric. 
\end{rem}

\subsection{Bisectors and spinal spheres}

Metric bisectors are important hypersurfaces of $\h$ because they naturally appear in the construction of Dirichlet fundamental domains. 

Let $z,w\in \h$ be two distinct points. The \textit{bisector of $z$ and $w$} is defined as 

$$B(z,w)=\{x\in \h ~|~ d(x,z) =d(x,w) \}.$$

Although bisectors are not totally geodesic, they admit natural foliations by complex lines, called \textit{slices}, and by real planes called \textit{meridians} (see \cite[Section 5]{GoldmanComplex}).
A \textit{spinal sphere} is defined as the intersection of the closure of a bisector with the boundary $\partial \h$. A spinal sphere also admits two natural foliations, one by the boundaries of the slices and the second by the boundaries of the meridians. A bisector is a separating real hypersurface in $\h$ diffeomorphic to $\R^3$ while a spinal sphere is a separating real hypersurface in $\partial \h$ diffeomorphic to $S^2$ (see \cite[Corollary 5.1.3]{GoldmanComplex}).

We now use the Siegel model of $\h$ and identify $\partial \h$ with $\mathcal{H}\cup \infty$, the one-point compactification of the Heisenberg group. We say that a bisector $B$ or its spinal sphere $B^*$ is \textit{finite} (resp. \textit{infinite}) if $\infty \notin B^*$ (resp. if $\infty \in B^*$).

The unit Cygan sphere provides an example of a finite spinal sphere while the complex plane provides an example of an infinite spinal sphere (see Figure \ref{Spinal}).

\begin{figure}[!h]
\centering
\includegraphics[width=0.4\textwidth]{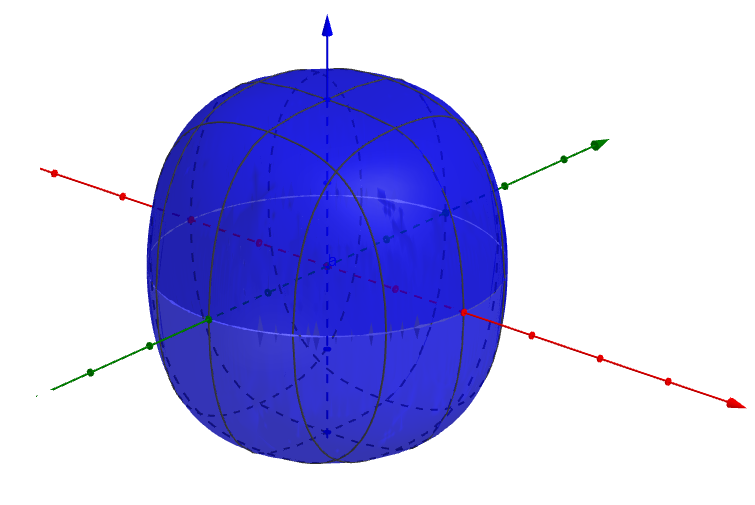}
\includegraphics[width=0.4\textwidth]{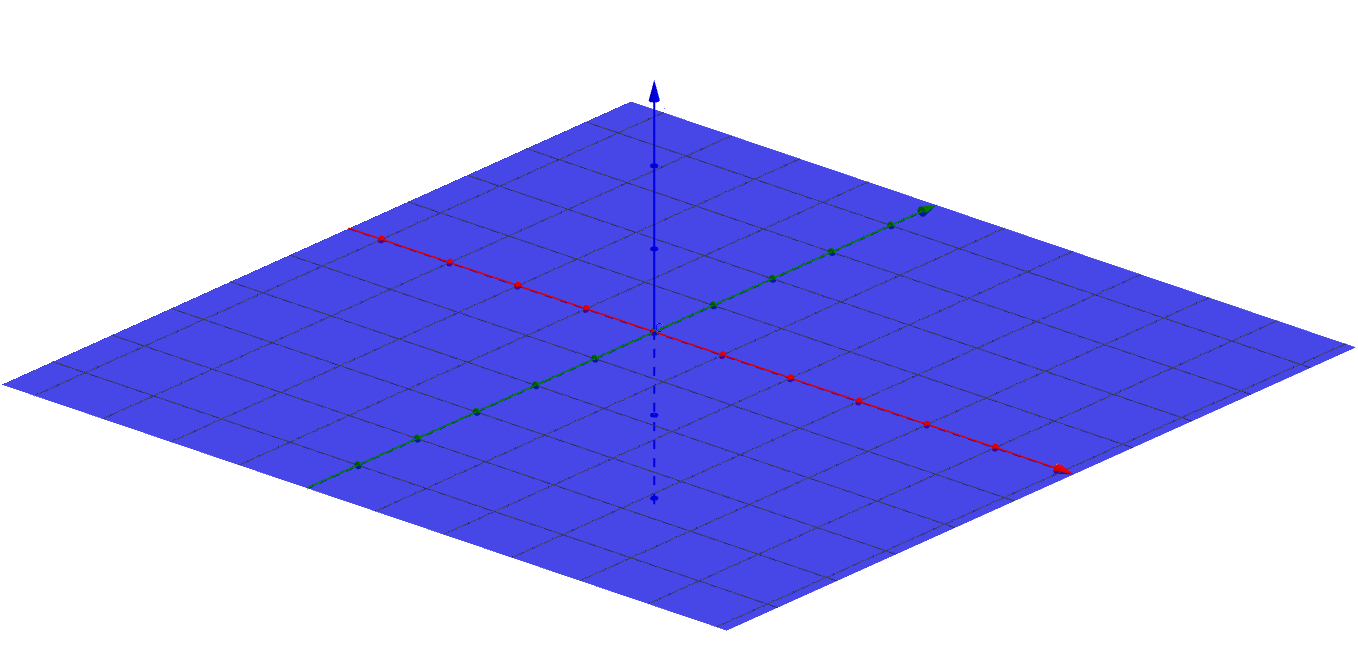}
\caption{The unit Cygan sphere defined by $(x^2+y^2)^2+z^2=1$ and the complex plane}
\label{Spinal}
\end{figure}

Two bisectors $B_1$ and $B_2$ are said to be \textit{coequidistant} if there are $z,w_1,w_2 \in \h$ such that $B_1=B(z,w_1)$ and $B_2=(z,w_2)$.

\begin{prop}\label{Bisectors}
If $B_1$ and $B_2$ are intersecting, infinite, coequidistant bisectors, then,\\ $\infty\in \overline{B_1 \cap B_2}\cap \partial \h$.
\end{prop}

\begin{proof}
Denote by $B^*_1$ and $B^*_2$ the spinal spheres of $B_1$ and $B_2$ respectively.
Goldman proved in \cite[Theorem 9.2.6]{GoldmanComplex} that the intersection of two coequidistant bisectors is connected. According to \cite[Lemma 9.1.1]{GoldmanComplex}, we can use \cite[Lemma 9.1.5]{GoldmanComplex} to deduce that $B_1\cap B_2$ is an open 2-disk and the discussion above \cite[Lemma 9.1.5]{GoldmanComplex} shows that $\overline{B_1\cap B_2}\setminus (B_1\cap B_2) \subset \partial \h$. 
We naturally have $\overline{B_1\cap B_2} \subset \overline{B_1}\cap \overline{B_2}$ with $\overline{B_1}\cap \overline{B_2}$ being a closed 2-disk again by \cite[Lemma 9.1.5]{GoldmanComplex}. Therefore $\overline{B_1\cap B_2}=\overline{B_1}\cap \overline{B_2}$ and $\infty\in \overline{B_1 \cap B_2}\cap \partial \h$.
\end{proof}

\subsection{Dirichlet fundamental polyhedra}

Let $\Gamma$ be a discrete subgroup of $\mathrm{PU}(2,1)$ and $\Upsilon$ be a subgroup of $\Gamma$. We suppose that $\Gamma$ has a left coset decomposition:
$$\Gamma=\bigsqcup_{n\in I} \gamma_n \Upsilon,$$
where $I$ is a subset of $\mathbb{N}$. We write $H=\bigcup_{n\in I} \{\gamma_n\}$ a set of left coset representatives.\\

A \textit{polyhedron} is a connected open subset of $\h$ defined as the intersection of countably many half-spaces bounded by bisectors. A polyhedron $D$ is a \textit{fundamental polyhedron for the coset decomposition $\Gamma / \Upsilon$} if: 

\begin{enumerate}[label=\arabic*)]
\item $D$ is $\Upsilon$-invariant,
\item There exists a set $\widetilde{D}$ with $D\subset \widetilde{D} \subset \overline{D}$ such that the images of $\widetilde{D}$ under elements of $H$ tesselate $\h$. 
\end{enumerate} 

A three-dimensional subset $S$ of a bisector defining $D$ is called a \textit{side} if there is $\gamma\in \Gamma$ such that $\gamma(\overline{D}) \cap \overline{D} = S$. We call such a $\gamma\in \Gamma$ a \textit{side pairing}.\\

An important example is the \textit{Dirichlet polyhedron centered at $o$} for some basepoint $o\in \h$. 
$$D_o(\Gamma)=\bigcap_{\gamma\in \Gamma\setminus\mathrm{Stab}_\Gamma(o)}\left\{x\in \h: d\bigl(x,o\bigl)< d\bigl(x,\gamma(o)\bigl) \right \}$$

The polyhedron $D_o(\Gamma)$ is a fundamental polyhedron for the coset decomposition $\Gamma / \mathrm{Stab}_\Gamma(o)$, where $\mathrm{Stab}_\Gamma(o)$ is the stabilizer of $o$ in $\Gamma$. Since $\Gamma$ acts properly discontinuously on $\h$, $\mathrm{Stab}_\Gamma(o)$ is necessarily a finite subgroup of $\Gamma$. Note that any two bisectors defining $D_o(\Gamma)$ are coequidistant.\\

We say that a fundamental polyhedron $D$ is \textit{locally finite} if every $x\in \h$ has a neighbourhood that meets only finitely many $\Gamma$-translates of $D$, and that $D$ is \textit{star-shaped (about $o$)} if there is $o\in D$ such that for all $x\in D$, the geodesic segment $[o,x]$ is contained in $D$.

\begin{prop}\label{Dirichlet}
A Dirichlet fundamental polyhedron based at $o$ is star-shaped about $o$ and locally finite.
\end{prop}

\begin{proof}
Denote by $D=D_o(\Gamma)$ the Dirichlet fundamental polyhedron based at $o$. Since the half-space bounded by the bisector $B\bigl(o,\gamma(o)\bigl)$ is star-shaped about $o$ for every $\gamma\in \Gamma$, $D$ is also star-shaped about $o$.\\
To prove that $D$ is locally finite, we follow \cite[Theorem 9.4.2]{Beardon} and prove that every closed ball $B(o,R)$ of radius $R$ centered at $o$ intersects only finitely many images of $D$. If $\gamma \in \Gamma$ satisfies $\gamma(D) \cap B(o,R) \neq \emptyset$, then there is some $z \in D$ such that $d(\gamma(z),o) \leqslant R$. Moreover, if $\gamma \in \Gamma \setminus \mathrm{Stab}_\Gamma(o)$, then $\gamma(z) \notin D$ and $d(z,o) \leqslant d(\gamma(z),o)$, therefore :

\begin{equation*} 
\begin{split}
d\bigl(o,\gamma (o)\bigl) & \leqslant d\bigl(o, \gamma(z)\bigl) + d\bigl(\gamma(z),\gamma(o)\bigl) \\
 & \leqslant R + d(z,o) \\
 & \leqslant R + d(\gamma(z),o) \\
 & \leqslant 2R.
\end{split}
\end{equation*}

Since $\Gamma$ acts properly discontinuously on $\h$ and $\mathrm{Stab}_\Gamma(o)$ is finite, there are only finitely many $\gamma\in \Gamma$ such that $\gamma(D) \cap B(o,R) \neq \emptyset$.
\end{proof}

\subsection{A Beardon-Maskit type theorem}

We now aim at proving the following theorem proved by Beardon and Maskit in \cite{BM} for discrete subgroups of $\mathrm{PSL}_2(\C)$.

\begin{thm}\label{BM}
Let $\Gamma$ be a discrete subgroup of $\mathrm{PU}(2,1)$. If $\Gamma$ admits a Dirichlet fundamental polyhedron with finitely many sides, then $\Gamma$ is geometrically finite.
\end{thm}

\begin{rem}
This result is part of the mathematical folklore and is expected to hold in much greater generality. Since it plays an important role for our purposes, we include a proof here. We believe this theorem to be true not only for discrete subgroups of $\mathrm{PU}(2,1)$, but as our argument uses the technical Proposition \ref{Bisectors} concerning bisectors in $\h$, we postpone the general case to another occasion. Nevertheless, the proof adapts easily to discrete subgroups of $\mathrm{Isom}(\mathbb{H}^n_{\R}) \simeq \mathrm{PO}(n,1)$.
\end{rem}

Our argument closely follows the classical proof of Beardon and Maskit in \cite{BM} for $\mathrm{PSL}_2(\mathbb{C})$, but we provide additional details at several steps. Before turning to the technical part, we outline the structure of the proof, indicate where the geometry of $\h$ introduces specific difficulties, and explain why the argument extends with only minor modifications to $\mathrm{PO}(n,1)$, $n\ge 2$.\\

Let $\Gamma$ be a discrete subgroup of $\mathrm{PU}(2,1)$ and $D$ be a Dirichlet fundamental polyhedron with finitely many sides. We set
$$D^*= \partial \h \cap \overline{D},$$
the intersection of $\partial \h$ with the closure of $D$ in $\overline{\h}$.

\begin{itemize}

\item[\textbf{(1)}]  
Using the local finiteness of $D$, we first show that if $D$ has finitely many sides, then no point of $D^*\cap \Lambda(\Gamma)$ can be conical.

\item[\textbf{(2)}]  
Next, using the fact that $D$ is star-shaped, we prove that if $z\in D^* \cap \Lambda(\Gamma)$, then the stabilizer $\mathrm{Stab}_{\Gamma}(z)$ is infinite. Combined with Step~(1), this implies that such a point $z$ must be a parabolic fixed point.

\item[\textbf{(3)}]  
The most delicate part is to show that any $z\in D^* \cap \Lambda(\Gamma)$ is in fact a \textbf{bounded} parabolic fixed point. To do so, we use the Siegel model and conjugate so that $z=\infty$. The set $\partial \h \setminus \{\infty\}$ identifies with the Heisenberg space $\mathcal{H}$. We show that $\mathcal{H}/ \mathrm{Stab}_\Gamma(\infty)$ is contained in the union of finitely many compact "balls" which arise as boundaries of half-spaces bounded by finite bisectors. 

In this process, the only place where the geometry of $\h$ appears to present a difficulty is in the following lemma:

\medskip
\emph{If $z\in D^*$ and $S$ is a side of $D$, contained in a bisector $B$ such that $z \in \overline{B}\cap \partial\h$, then $z \in \overline{S}\cap \partial\h$.}
\medskip

This statement appears as Claim~2 in Lemma~\ref{bpf}. It is a consequence of Proposition~\ref{Bisectors} about bisectors in $\h$. However, the analog of Proposition~\ref{Bisectors} also holds for bisectors in real hyperbolic spaces (the proof is simpler because bisectors are totally geodesic). Thus, upon replacing the Siegel model by the upper half-space model and the Heisenberg space by the Euclidean space, the argument carries over to $\mathrm{PO}(n,1)$ with the exact same ideas.

\item[\textbf{(4)}]  
Finally, we use the dynamics of bounded parabolic fixed points to show that every element of $\Lambda(\Gamma)$ which is not in the $\Gamma$-orbit of some $z\in D^* \cap \Lambda(\Gamma)$ is necessarily conical.

\end{itemize}

Before starting the proof of Theorem \ref{BM}, we give another characterization of conical points. Let $\Gamma$ be a discrete subgroup of $\mathrm{PU}(2,1)$.

\begin{lem}\label{con}
An element $z\in \Lambda(\Gamma)$ is conical if and only if for some/each geodesic ray $\beta$ ending at $z$, there is a sequence $(\gamma_n)_{n\in\mathbb{N}}$ of distinct elements of $\Gamma$ and a relatively compact subset $K\subset \h$ such that $\gamma_n(\beta)\cap K \neq \emptyset$ for all $n\in \mathbb{N}$.
\end{lem}

\begin{proof}
Let $z$ be a conical point: there are distinct $z',z''\in \Lambda(\Gamma)$ and $(\gamma_n)_{n\in\mathbb{N}}$ a convergence sequence with base $(z,z')$ such that $\gamma_n(z) \rightarrow z''$. Let $o\in \h $ be a basepoint and $\beta=[o,z)$ be a geodesic ray ending at $z$. Then $\bigl(\gamma_n(\beta)\bigl)_{n\in\mathbb{N}}=\bigl(\bigl[\gamma_n(o),\gamma_n(z)\bigl)\bigl)_{n\in\mathbb{N}}$ is a sequence of geodesic rays with $\gamma_n(o)\rightarrow z'$ and $\gamma_n(z) \rightarrow z''$. Therefore, we can take $K$ to be a ball of radius $r$ centered on a point of the geodesic $(z',z'')$ for some $r>0$. Conversely, suppose that $z$ is not conical so that every convergence sequence $(\gamma_n)_{n\in\mathbb{N}}$ with base $(z,z')$ satisfies $\gamma_n(z)\rightarrow z'$. In that case, $\bigl(\gamma_n(\beta)\bigl)_{n\in\mathbb{N}}=\bigl(\bigl[\gamma_n(o),\gamma_n(z)\bigl)\bigl)_{n\in\mathbb{N}}$ is a sequence of geodesic rays with both $\gamma_n(o)\rightarrow z'$ and $\gamma_n(z) \rightarrow z'$. Therefore, the sequence of Gromov products $\langle \gamma_n(o),\gamma_n(z)\rangle_o \rightarrow_n +\infty$, meaning that the distance between the geodesic ray $\gamma_n(\beta)$ and the basepoint $o$ goes to $\infty$.
\end{proof}

Let $\Gamma$ be a discrete subgroup of $\mathrm{PU}(2,1)$ and $o\in \h$ be a basepoint such that the Dirichlet fundamental polyhedron $D=D_o(\Gamma)$ has finitely many sides. We denote by $J_o=\mathrm{Stab}_\Gamma(o)$ the stabilizer of $o$ in $\Gamma$. Recall that $D$ is a fundamental polyhedron for the coset decomposition $\Gamma / J_o$. We write $D^*= \partial \h \cap \overline{D}$ the intersection of $\partial \h$ with the closure of $D$ in $\overline{\h}$.

\begin{lem}\label{ncon}
If $D$ has finitely many sides, the set $D^*$ does not contain any conical point.
\end{lem}

\begin{proof}
Recall that $D$ is star-shaped about $o$ and locally finite, by Proposition \ref{Dirichlet}.
If $z\in D^*$, the geodesic ray $\beta=[o,z)$ ending at $z$ is contained in $D$. Suppose that $z$ is conical. Using Lemma \ref{con}, there is a sequence $(\gamma_n)_{n\in\mathbb{N}}$ of distinct elements of $\Gamma$ and a relatively compact subset $K\subset \h$ such that $\gamma_n(\beta)\cap K \neq \emptyset$ for all $n\in \mathbb{N}$. This contradicts the fact that $D$ is locally finite.
\end{proof}

We deduce:

\begin{prop}\label{bpf}
If $D$ has finitely many sides, every element of $D^*\cap \Lambda(\Gamma)$ is a bounded parabolic fixed point.
\end{prop}

\begin{proof}
Let $z$ be an element of $D^*$. Since $D$ has finitely many sides and $J_o$ is a finite subgroup of $\Gamma$, there is only a finite number of elements $z=z_0, z_1, ..., z_m$ in $D^* \cap \Gamma(z)$. Let $\gamma_0, \gamma_1, ... , \gamma_m$ be elements of $\Gamma$ such that, for all $i\in \{0,...,m\}$, $z_i=\gamma_i^{-1}(z)$, where for convenience, $\gamma_0=\mathrm{Id}$.\\

We denote by $G_z=\{\gamma\in \Gamma ~\vert ~ z\in \gamma(D^*) \}$ and by $J_z=\mathrm{Stab}_\Gamma(z)$ the stabilizer of $z$ (see Figure \ref{bmm1}).\\

\textbf{\underline{Claim 1}:} $G_z \subset \bigcup_{i=0}^m J_z\gamma_i$. 

\begin{proof}
Let $g\in G_z$. Since $D$ is star-shaped about $o$, the geodesic ray $\beta=[g(o),z)$ is contained in $g(D)$. Then $g^{-1}(\beta)=\bigl[o,g^{-1}(z)\bigl)$ is a geodesic ray contained in $D$ ending at an element in both the $\Gamma$-orbit of $z$ and $D^*$. Hence, there is $i\in \{0,...,m\}$ such that $\gamma_ig^{-1}(\beta)$ is a geodesic ray ending at $z$. This implies that $\gamma_ig^{-1}\in J_z$, hence $g\in J_z\gamma_i$.
\end{proof}

\begin{figure}[!h]
\centering
\includegraphics[width=0.7\textwidth]{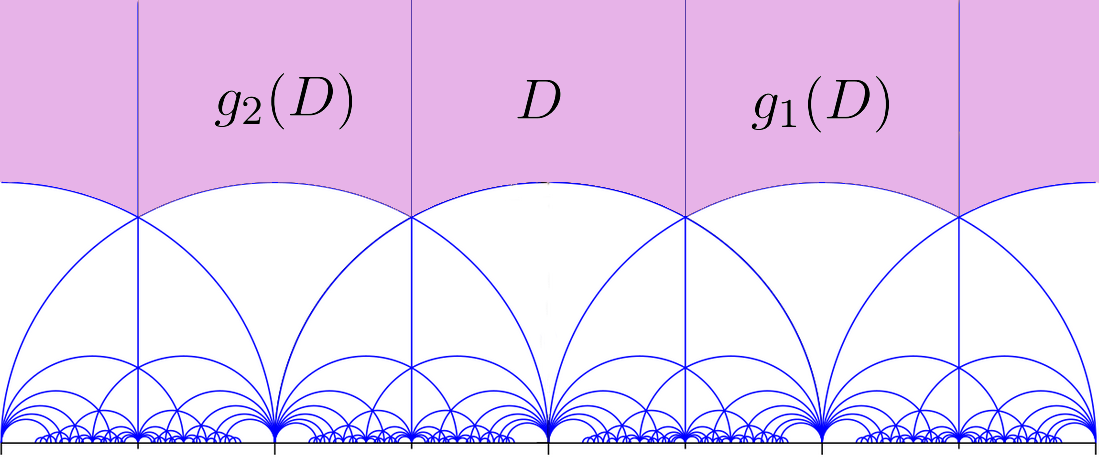}
\caption{A fundamental domain $D$ and its image under the elements of $G_z$ ($z=\infty$)}
\label{bmm1}
\end{figure}

The set $G_z$ must be infinite. Indeed, if $G_z$ is finite, there is a neighbourhood of $z$ in $\overline{\h}$ which meets only finitely many $\Gamma$-translates of $D$, contradicting the fact that $z\in \Lambda(\Gamma)$. Using Claim 1, we deduce that $J_z$ also has infinite order, therefore, it contains an infinite order element. This element cannot be elliptic because $\Gamma$ is a discrete subgroup of $\mathrm{PU}(2,1)$ and cannot be hyperbolic by Lemma \ref{ncon}. Hence, $J_z$ contains a parabolic element, and $z$ is a parabolic fixed point.\\

In order to show that $z$ is a bounded parabolic fixed point, we suppose, without loss of generality that $z=\infty$ so that $\partial \h \setminus \{z\}$ identifies with the Heisenberg group $\mathcal{H}$.\\

\textbf{\underline{Claim 2}:} If $\infty\in D^*$ and $S$ is a side of $D$ contained in an infinite bisector, then $\infty\in \overline{S}$.

\begin{proof}
Let $S$ be a side contained in an infinite bisector $B$ and $S'$ be another side of $D$ intersecting $S$. The side $S'$ is contained in a bisector $B'$ intersecting $B$. If $B'$ is finite, since $\infty\in D^*$, both $D$ and $S$ are contained in the connected component of $\h \setminus B'$ which contains $\infty$ in its boundary. If $B'$ is infinite, then $B$ and $B'$ form a pair of intersecting, infinite and coequidistant bisectors, hence $\infty\in \overline{B \cap B'}\cap \partial \h$ by Proposition \ref{Bisectors}.
\end{proof}

If a bisector $B$ is finite, we call the connected component of $\h \setminus B$ which does not contain $\infty$ in its boundary the \textit{finite half-space bounded by $B$}. Its boundary in $\overline{\h}$, that we call \textit{spinal ball}, is the union of the finite spinal sphere $B^*$ with the connected component of $\mathcal{H} \setminus B^*$ which does not contain $\infty$. Spinal balls are bounded with respect to the Cygan metric, and hence compact subsets of $\mathcal{H}$.\\

Recall that $J_{\infty}=\mathrm{Stab}_{\Gamma}(\infty)$ and $G_\infty=\{\gamma\in \Gamma ~\vert ~ \infty\in \gamma(D^*) \}$. Let $\gamma\in \Gamma$ be such that $\gamma(D)$ has a side contained in an infinite bisector. According to Claim 2, we have $\gamma\in G_\infty$ and by Claim 1, $\gamma \in \bigcup_{i=0}^m J_{\infty}\gamma_i$. \\

For each $i\in \{0,...m\}$, we denote by $K_i$ the union of the finite half-spaces bounded by the finite bisectors defining $\gamma_i(D)$, and we let $K=\bigcup_{i=0}^m K_i$ (see Figure \ref{bmm2}).

\begin{figure}[!h]
\centering
\includegraphics[width=0.7\textwidth]{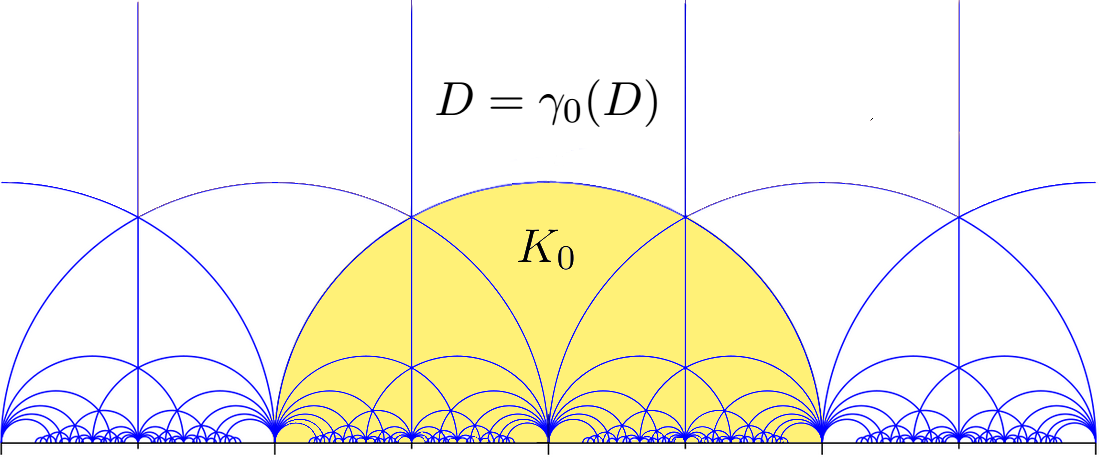}
\caption{A finite half-space bounded by a finite bisector}
\label{bmm2}
\end{figure}

Now consider $D_\infty=D_o(J_\infty)$ the Dirichlet fundamental polyhedron centered at $o$ for $J_\infty$. We have $D \subset D_\infty$ and 
$$D_\infty \setminus \bigcup_{i=0}^m \gamma_i(D) \subset K.$$

Hence, if $z'\in \Lambda(\Gamma) \setminus \{\infty\}$, then there exists $j \in J_\infty$ such that $j(z')$ lies in $K^*=\overline{K}\cap \partial \h$. Since $K^*$ is a finite union of spinal balls, it is a compact subset of $\mathcal{H}$. Hence $J_\infty$ acts cocompactly on $\Lambda(\Gamma) \setminus \{\infty\}$.
\end{proof}

We are now ready to prove Theorem \ref{BM}:

\begin{proof}[Proof of Theorem \ref{BM}]
Let $z\in \Lambda(\Gamma)$ be an element of the limit set, $o\in \h$ be a basepoint and $\beta=[o,z)$ be a geodesic ray of $\h$ ending at $z$. If $\beta$ intersects only finitely many $\Gamma$-images of $D$, then $z$ is a bounded parabolic fixed point according to Proposition \ref{bpf}.\\

Since $D$ has finitely many sides, if $\beta$ intersects infinitely many $\Gamma$-images of $D$, there exists a side $S$ of $D$ and a sequence $(\gamma_n)_{n\in \mathbb{N}}$ of distinct elements of $\Gamma$ such that $\gamma_n(S) \cap \beta \neq \emptyset$ for all $n\in \mathbb{N}$. By Theorem \ref{Tukia}, there exist $a,b\in \Lambda(\Gamma)$ such that $(\gamma_n)_{n\in \mathbb{N}}$ is a convergence sequence with base $(a,b)$. \\

For each $n\in \mathbb{N}$, let $y_n\in S\cap \gamma_n^{-1}(\beta)$ (see Figure \ref{BMf}). If $(y_n)_{n\in\mathbb{N}}$ has an accumulation point in $\h$, then $z$ is conical according to Lemma \ref{con}. If not, then up to subsequence, $(y_n)_{n\in\mathbb{N}}$ converges to $y\in D^*$. Since $y_n\in \gamma_n^{-1}(\beta)=\bigl[\gamma_n^{-1}(o),\gamma_n^{-1}(z)\bigl)$, we must have $\gamma_n^{-1}(o) \rightarrow y$ or $\gamma_n^{-1}(z) \rightarrow y$. Since $(\gamma_n^{-1})_{n\in \mathbb{N}}$ is a convergence sequence with base $(b,a)$, we have $\gamma_n^{-1}(o) \rightarrow y$ if and only if $a=y$.

\begin{figure}[!h]
\centering
\includegraphics[width=0.6\textwidth]{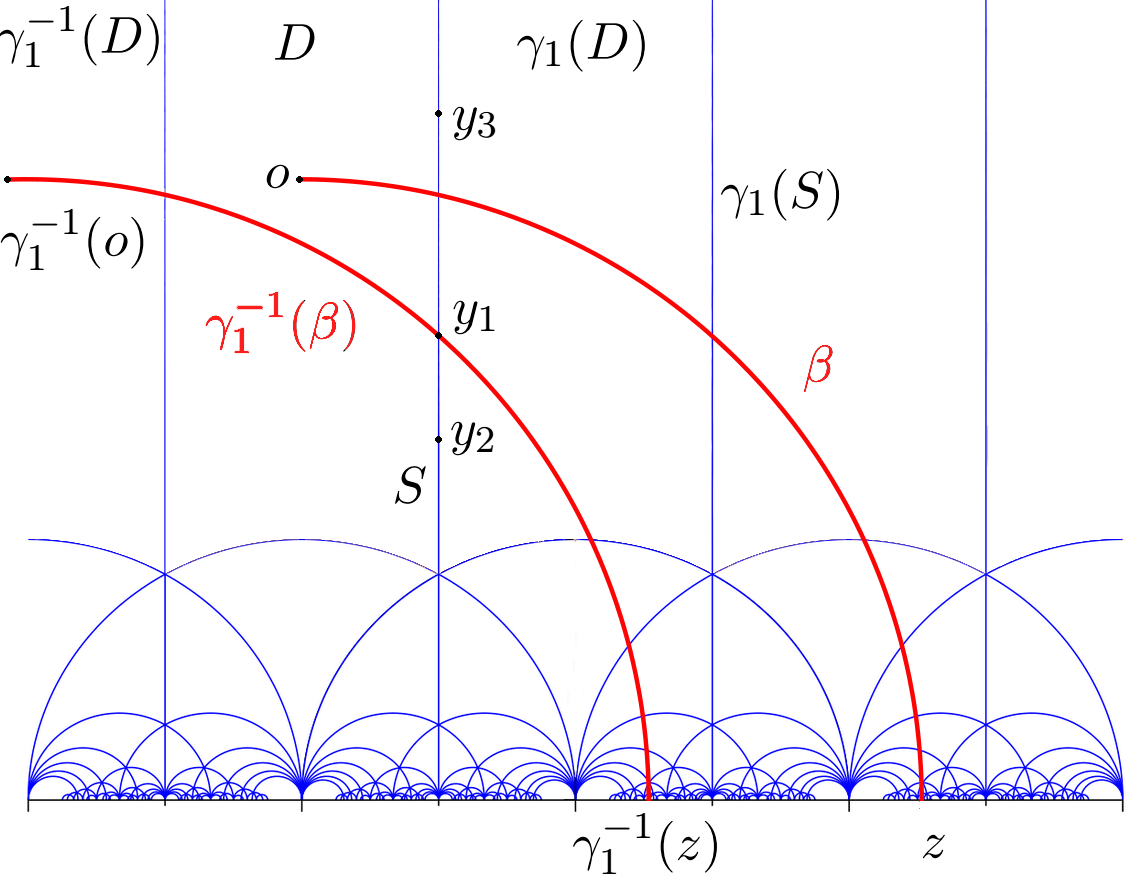}
\caption{A geodesic ray $\beta$, the side $S$ and the sequence $(y_n)_{n\in \mathbb{N}}$}
\label{BMf}
\end{figure}

First, suppose that $a\neq y $ and $\gamma_n^{-1}(z) \rightarrow y$.\\
Since $y_n \rightarrow y$ and $a\neq y$, the sequence $(y_n)_{n\in\mathbb{N}}$ is contained in a compact subset of $\overline{\h}\setminus \{a\}$, therefore $\gamma_n(y_n)\rightarrow b$. Moreover, for all $n\in \mathbb{N}$, $\gamma_n(y_n)\in \beta=[o,z)$ . We deduce that $b=z$ and $(\gamma_n^{-1})_{n\in \mathbb{N}}$ is a convergence sequence with base $(z,a)$ such that $\gamma_n^{-1}(z)\rightarrow y$ and $y\neq a$. Hence $z$ is a conical point.\\

Suppose now that $a=y$ so that $(\gamma_n)_{n\in \mathbb{N}}$ is a convergence sequence with base $(y,b)$.\\
Since $y\in D^*$, it is a bounded parabolic fixed point by Proposition \ref{bpf}. If $\gamma_n^{-1}(z)=y$ for some $n\in \mathbb{N}$, then $z$ is also a bounded parabolic fixed point in the $\Gamma$-orbit of $y$. If not, there is a compact subset $K$ of $\Lambda(\Gamma)\setminus \{y\}$ and a sequence $(j_n)_{n\in \mathbb{N}}$ of elements of $\mathrm{Stab}_\Gamma(y)$ such that $j_n\bigl(\gamma_n^{-1}(z)\bigl)\in K$ for all $n\in \mathbb{N}$. Up to subsequence, we can suppose that there is $y'\in K$ with $y'\neq y$ such that $j_n\bigl(\gamma_n^{-1}(z)\bigl)\rightarrow y'$. Moreover, $\gamma_n^{-1}(o)\rightarrow y$, hence $j_n \circ \gamma_n^{-1}(o) \rightarrow y$ by Lemma \ref{Busee}. Therefore, $(j_n\circ \gamma_n^{-1})_{n\in \mathbb{N}}$ is a convergence sequence with base $(z,y)$ such that $j_n\bigl(\gamma_n^{-1}(z)\bigl)\rightarrow y'$ and $y'\neq y$. Hence $z$ is a conical point.

\end{proof}

\section{Triangle groups}

In this section, we define abstract triangle groups and describe some of their representations in $\mathrm{PU}(2,1)$. We exhibit a finite-index subgroup isomorphic to the fundamental group of a closed surface and describe the nature of a particular element.

\subsection{Definition}

Let $p,q,r$ be positive integers greater than one. The abstract $(p,q,r)$ triangle group $\Delta_{p,q,r}$ is defined by the following presentation:

$$\Delta_{p,q,r}=\Bigl \langle i_1,i_2,i_3 ~|~ i_1^2=i_2^2=i_3^2=(i_2i_3)^p=(i_3i_1)^q=(i_1i_2)^r
\Bigr \rangle.$$

This group can be realized geometrically as a group of isometries of the 2-sphere, the Euclidean plane or the real hyperbolic plane. To construct such a realization, we start with three pairwise concurrent geodesics $L_1, L_2$ and $L_3$ such that $$\angle(L_2,L_3)=\pi/p, \quad \angle(L_1,L_3)=\pi/q, \quad \angle(L_1,L_2)=\pi/r.$$

The triangle $T$ formed by these three geodesics lies in the sphere $S^2$, the Euclidean plane $\mathbb{E}^2$ or the real hyperbolic plane $\mathbb{H}_{\mathbb{R}}^2$, depending on whether $1/p+ 1/q + 1/r$ is greater, equal to, or less than 1, respectively. In each of these spaces, the reflection in a geodesic is a well-defined isometry. The group generated by the three reflections $I_1$, $I_2$ and $I_3$ in $L_1$, $L_2$ and $L_3$ is a discrete subgroup of $\mathrm{Isom}(S^2)$ (resp. $\mathrm{Isom}(\mathbb{E}^2)$, $\mathrm{Isom}(\mathbb{H}_{\mathbb{R}}^2)$), with fundamental domain $T$, and isomorphic to $\Delta_{p,q,r}$.\\

Suppose that $1/p+ 1/q + 1/r<1$. Since a hyperbolic triangle is uniquely determined up to isometry by its angles, there is a unique conjugacy class of discrete and faithful representations of $\Delta_{p,q,r}$ in $\mathrm{Isom}(\mathbb{H}_{\mathbb{R}}^2)$.\\

We also allow the case where one or more of $p$, $q$, $r$ is infinite. In such cases, we omit the corresponding relator in the previous presentation and replace the concurrent geodesics by asymptotic ones.

\begin{figure}[!h]
\centering
\includegraphics[width=0.45\textwidth]{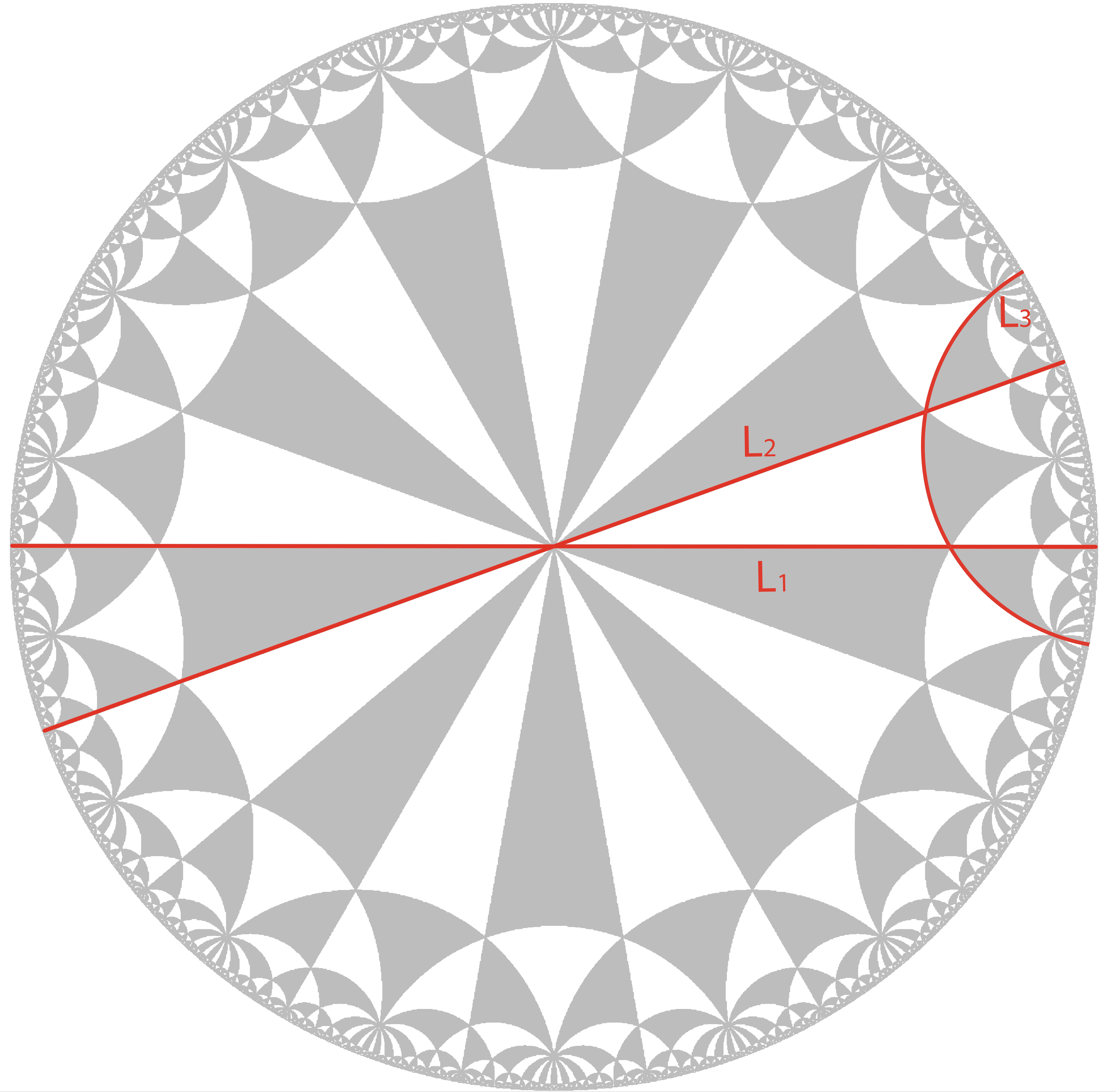}
\caption{A geometric realization of $\Delta_{3,3,9}$}
\label{3,3,4 vierge}
\end{figure}

\subsection{Complex hyperbolic triangle groups}

Let $p,q,r$ be positive integers such that $p\leqslant q \leqslant r$ and $1/p+ 1/q + 1/r<1$.\\

A \textit{$(p,q,r)$-complex hyperbolic triangle} (or just $(p,q,r)$-\textit{triangle}) is a triple $(L_1, L_2,L_3)$ of pairwise concurrent complex lines of $\h$ such that $$\angle(L_2,L_3)=\pi/p, \quad \angle(L_1,L_3)=\pi/q, \quad \angle(L_1,L_2)=\pi/r.$$
As before, we allow one or more of $p$, $q$, $r$ to be infinite. In such cases, we replace the concurrent complex lines by asymptotic ones.\\

In contrast to real hyperbolic geometry, when $ p\geqslant 3$, there is a (real) one-parameter family of non-isometric $(p,q,r)$-triangles. In order to describe this family by a real invariant, we make the following observation.\\

There is a bijection between the set of complex lines of $\h$ and the set of positive lines in $\mathbb{C}^{2,1}$. Indeed, to any positive line, one can associate its orthogonal with respect to $\langle \cdot, \cdot \rangle$ which is a complex $2$-plane of $\mathbb{C}^{2,1}$ whose projectivization intersects $\h$. Reciprocally, each complex $2$-plane of $\mathbb{C}^{2,1}$ intersecting $\h$ can be viewed as the kernel of a linear form dual to a positive vector of $\mathbb{C}^{2,1}$.\\
If $L\subset \mathbb{H}^2_{\mathbb{C}}$ is a complex line, such a positive vector is called \textit{polar vector to $L$} and is unique up to scaling. A polar vector $l$ is called \textit{normalized} if $\langle l,l \rangle=1$. \\

The \textit{angular invariant} $\alpha$ of $(L_1,L_2,L_3)$ is defined as:
$$\alpha:= \arg \left(\prod_{k=1}^{3} \langle l_{k-1}, l_{k+1} \rangle \right),$$
where $l_k$ denotes a normalized polar vector to $L_k$ (indices are taken modulo 3). Note that changing some $l_k$ by $-l_k$ does not change the value of $\alpha$.

Pratoussevitch proved the following result:

\begin{prop}[{\cite[Section 3]{Pratoussevitch}}]\label{T1}
A $(p,q,r)$-triangle in $\h$ is determined uniquely, up to holomorphic isometry, by the triple $(p,q,r)$ and the angular invariant $\alpha\in [0,2\pi]$.\\
For any $\alpha \in [0,2\pi)$, there exists a $(p,q,r)$ triangle with angular invariant $\alpha$ if and only if 
$$\mathrm{cos}(\alpha) < \frac{c_1^2+c_2^2+c_3^2-1}{2c_1c_2c_3}$$

where $(c_1,c_2,c_3)= \bigl(\mathrm{cos}(\frac{\pi}{2p}),\mathrm{cos}(\frac{\pi}{2q}),\mathrm{cos}(\frac{\pi}{2r}) \bigl)$.
\end{prop}
For $\alpha \in [0,2\pi)$, we denote by a $(p,q,r;\alpha)$-triangle a $(p,q,r)$-triangle with angular invariant $\alpha$. \\

Given a complex line $L$ with polar vector $l$, the \textit{inversion in $L$} is the isometry of $\h$ induced by the linear transformation of $\mathbb{C}^{2,1}$:
\begin{equation*}
I_L: z \rightarrow -z+2\frac{\langle z, l \rangle}{\langle l,l \rangle}l.
\end{equation*}

It is a complex reflection which is the unique isometry of $\h$ of order 2 whose fixed point set is equal to $L$.\\

Let $(L_1,L_2,L_3)$ be a $(p,q,r;\alpha)$-triangle and for $k\in \{1,2,3\}$, let $I_k\in \mathrm{PU}(2,1)$ denote the inversion in $L_k$. The group $\langle I_1,I_2,I_3\rangle $ generated by $I_1,I_2$ and $I_3$ is called a \textit{$(p,q,r;\alpha)$-triangle group}. It is the image of a representation of $\Delta_{p,q,r}$ in $\mathrm{PU}(2,1)$ that we call a \textit{$(p,q,r;\alpha)$-representation}. In contrast with representations in $\mathrm{Isom}(\mathbb{H}^2_{\R})$, not every $(p,q,r;\alpha)$-representation must be discrete and faithful.\\

A $(p,q,r;\pi)$-triangle group always exists and preserves a real plane (see \cite[Section 3]{Pratoussevitch}). It yields a discrete and faithful representation of $\Delta_{p,q,r}$ in $\mathrm{PO}(2,1)\subset \mathrm{PU}(2,1)$. Moreover, a $(p,q,r;\alpha)$-triangle group is conjugate to a $(p,q,r;2\pi-\alpha)$-triangle group by an anti-holomorphic isometry of $\h$, hence we can restrict our attention to $\alpha\in [0, \pi]$.

\begin{rem}
There are other interesting representations of $\Delta_{p,q,r}$ in $\mathrm{Isom}(\h)$ which can be obtained by mapping each generator to an isometry of $\h$ fixing pointwise a real plane. We shall not consider them here.
\end{rem}

Schwartz presented in \cite{Schwartz3} the following conjecture.

\begin{conj}
Let $p,q,r$ be positive integers such that $p\leqslant q \leqslant r$ and $1/p+ 1/q + 1/r<1$.\\
A $(p,q,r;\alpha)$ representation of $\Delta_{p,q,r}$ is discrete and faithful if and only if $W_A=I_1I_3I_2I_3$ and $W_B=I_1I_2I_3$ are non-elliptic. Furthermore:
\begin{itemize}
\item If $p<10$ then the representation is discrete and faithful if and only if $W_A=I_1I_3I_2I_3$ is non-elliptic.
\item If $p>13$ then the representation is discrete and faithful if and only if $W_B=I_1I_2I_3$ is non-elliptic.
\end{itemize}
\end{conj}

Thanks to Theorem \ref{trace} and the following proposition due to Pratoussevitch, we can express the previous conditions in terms of the angular invariant.
\begin{samepage}
\begin{prop}[{\cite[Section 8]{Pratoussevitch}}]\label{T2}
Let $p,q,r$ be positive integers such that $p\leqslant q \leqslant r$ and $1/p+ 1/q + 1/r<1$. Let $\alpha \in [0, \pi]$ and consider a $(p,q,r;\alpha)$-triangle group.\\

There are lifts of $W_A=I_1I_3I_2I_3$ and $W_B=I_1I_2I_3$ to $\mathrm{SU}(2,1)$, denoted by $\textbf{W}_\textbf{A}$ and $\textbf{W}_\textbf{B}$, whose traces are given by:
$$\mathrm{tr}(\textbf{W}_\textbf{A})=16c_2^2c_3^2+4c_1^2-1-16c_1c_2c_3\cos(\alpha)$$ 
$$\mathrm{tr}(\textbf{W}_\textbf{B})=8c_1c_2c_3e^{i\alpha} - 4c_1^2 - 4c_2^2 - 4c_3^2 -3$$

where $(c_1,c_2,c_3)= \bigl(\mathrm{cos}(\frac{\pi}{2p}),\mathrm{cos}(\frac{\pi}{2q}),\mathrm{cos}(\frac{\pi}{2r}) \bigl)$.
\end{prop}
\end{samepage}

The first steps toward proving Schwartz's conjecture have been taken.  

\begin{thm}[\cite{Goldman},\cite{Schwartz1},\cite{Schwartz4}]
An $(\infty,\infty,\infty;\alpha)$ representation of $\Delta_{\infty,\infty,\infty}$ is discrete and faithful if and only if $W_B=I_1I_2I_3$ is non-elliptic.
\end{thm} 

Schwartz also proved his conjecture for large values of $p$.

\begin{thm}[\cite{Schwartz2}]
Let $p,q,r$ be positive integers such that $p\leqslant q \leqslant r$ and $1/p+ 1/q + 1/r<1$.\\
If $p$ is large enough, a $(p,q,r;\alpha)$ representation of $\Delta_{p,q,r}$ is discrete and faithful if and only if $W_B=I_1I_2I_3$ is non-elliptic.
\end{thm}

Focusing on small values of $p$, Parker, Wang and Xie proved:

\begin{thm}[{\cite[Theorem 1.6]{Parker1}}]\label{Parker}
Let $n$ be an integer at least 4 and $\alpha\in [0,\pi]$. A $(3,3,n;\alpha)$ representation of $\Delta_{3,3,n}$ is discrete and faithful if and only if $W_A=I_1I_3I_2I_3$ is non-elliptic.
\end{thm}

We can reformulate the condition that $W_A$ is non-elliptic in a $(3,3,n; \alpha)$ representation in terms of the angular invariant $\alpha$. For every integer $n$ at least 4, we define: 
$$c_n=\cos \left(\frac{\pi}{2n}\right), ~~~~\alpha^{\min}_n= \cos^{-1} \left(\frac{2c_n^2+1}{3c_n}\right), ~~~~\alpha_n^0=\cos^{-1} \left(\frac{12c_n^2-1}{12c_n}\right).$$

A study of the functions displayed on Figure \ref{courbes} shows that $\alpha_n^0$ and $\alpha^{\min}_n$ are well-defined, and $$\alpha^{\min}_n < \alpha_n^0 < \pi.$$

\begin{figure}[!h]
\centering
\includegraphics[width=0.45\textwidth]{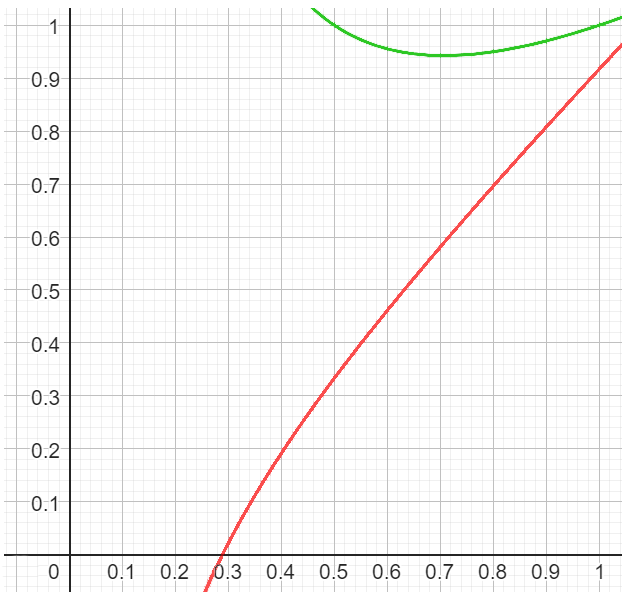}
\caption{The representative curves of $x\mapsto \frac{2x^2+1}{3x}$ in green and $x \mapsto \frac{12x^2-1}{12x}$ in red.}
\label{courbes}
\end{figure}

\begin{samepage}
We give a more accurate description of the situation:

\begin{prop}\label{final}
Let $n$ be an integer at least 4. 
\begin{itemize}
\item For every $\alpha\in (\alpha^0_n, \pi]$, a $(3,3,n; \alpha)$ representation of $\Delta_{3,3,n}$ in $\mathrm{PU}(2,1)$ is discrete, faithful and convex-cocompact.
\item For $\alpha=\alpha_n^0$, a $(3,3,n; \alpha_n^0)$ representation of $\Delta_{3,3,n}$ in $\mathrm{PU}(2,1)$ is discrete, faithful, and geometrically finite. Moreover $W_A=I_1I_3I_2I_3$ is parabolic and all parabolic fixed points belong to the same $\langle I_1,I_2,I_3\rangle$-orbit.
\item For every $\alpha\in [\alpha^{\min}_n, \alpha_n^0)$, a $(3,3,n; \alpha)$ representation of $\Delta_{3,3,n}$ in $\mathrm{PU}(2,1)$ is not discrete and faithful.
\end{itemize}
\end{prop}

\end{samepage}

\begin{proof}
For every $\alpha\in [\alpha^{\min}_n, \pi]$, there exists a $(3,3,n;\alpha)$-representation by Proposition \ref{T1}. Proposition \ref{T2} together with Theorem \ref{trace} tells us that $W_A$ is hyperbolic when $\alpha>\alpha_n^0$, unipotent when $\alpha=\alpha_n^0$, and elliptic when $\alpha<\alpha_n^0$. Since $i_1i_3i_2i_3$ is an infinite order element of $\Delta_{3,3,n}$, if $\alpha\in [\alpha^{\min}_n, \alpha_n^0)$, then a $(3,3,n; \alpha)$-representation cannot be discrete and faithful.\\

If $\alpha\in [\alpha^0_n, \pi]$, then every $(3,3,n;\alpha)$-representation is discrete and faithful by Theorem \ref{Parker}. Denote by $\Delta_n^\alpha$ the image of a $(3,3,n;\alpha)$-representation. There exists a $(3,3,n;\alpha)$-triangle $(L_1,L_2,L_3)$ such that $\Delta_n^\alpha$ is generated by the inversions $I_k$ in $L_k$ for $k\in \{1,2,3\}$. We call $o\in \h$ the intersection point of $L_1$ and $L_2$. Note that $\mathrm{Stab}_{\Delta_n^\alpha}(o)=\langle I_1, I_2 \rangle$ is isomorphic to a dihedral group of order $2n$. The Dirichlet polyhedron $D=D_o(\Delta_n^\alpha)$ centered at $o$ has finitely many sides by \cite[Corollary 4.4]{Parker1}. Therefore $\Delta_n^\alpha$ is geometrically finite according to Theorem \ref{BM}.\\

When $\alpha\in (\alpha^0_n, \pi]$, then it follows from \cite[Theorem 4.3]{Parker1} that $\Delta_n^\alpha$ contains no parabolic elements. Therefore $\Delta_n^\alpha$ is geometrically finite and totally hyperbolic, hence convex-cocompact (see \cite{Bow1}).\\

When $\alpha=\alpha_n^0$, $W_A=I_1I_3I_2I_3$ is a non-trivial unipotent isometry, hence parabolic. Moreover, every parabolic fixed point of $\Delta_n^{\alpha_n^0}$ is in the orbit of a parabolic fixed point lying in $D^*= \overline{D} \cap \partial \h$. It follows from \cite[Lemma 4.2 and Theorem 4.3]{Parker1} that there are exactly $2n$ parabolic fixed points in $D^*$ that all belong to the same $\langle I_1,I_2\rangle$-orbit. Therefore, there exists a single $\langle I_1,I_2,I_3\rangle$-orbit of parabolic fixed points.
\end{proof}

\subsection{Surface subgroups and non-simple closed curves} \label{surfsubgrp}

We now return to abstract triangle groups and real hyperbolic geometry. We identify $\mathrm{Isom}(\mathbb{H}_{\mathbb{R}}^2)$ with $\mathrm{PGL}_2(\mathbb{R})$ and $\mathrm{Isom}^+(\mathbb{H}_{\mathbb{R}}^2)$ with $\mathrm{PSL}_2(\mathbb{R})$. Let $p,q,r$ be positive integers such that $p\leqslant q \leqslant r$ and $1/p+ 1/q + 1/r<1$. In order not to complicate the notations, we still denote by $\Delta_{p,q,r}$ the image of $\Delta_{p,q,r}$ by a discrete and faithful representation in $\mathrm{PGL}_2(\mathbb{R})$. We denote by $I_1,I_2$ and $I_3$ the generators of $\Delta_{p,q,r}$.\\

We seek to construct a finite-index subgroup of $\Delta_{p,q,r}$ isomorphic to the fundamental group of a closed surface of genus $g\geqslant 2$. We call it a \textit{surface subgroup}. The following theorem is quite standard, hence we refer to \cite{Series} for the missing definitions.

\begin{thm}\label{Poincaré}
Let $ P \subset \mathbb{H}^2 $ be a finite-sided convex polygon with no ideal vertices. Let $S\subset \mathrm{PSL}_2(\mathbb{R})$ be a set of side pairings such that the angle sum corresponding to each vertex cycle is exactly $2\pi$. Then the group $\Gamma$ generated by $S$ is discrete, $P$ is a fundamental domain for the action of $\Gamma$ on $\mathbb{H}^2$, and $\Gamma$ is isomorphic to the fundamental group of a closed surface of genus $g=\frac{\mathrm{Area}(P)}{4\pi}+1.$
\end{thm}

\begin{proof}
The fact that $\Gamma$ is a discrete subgroup, that $P$ is a fundamental domain and that $\Gamma$ is isomorphic to the fundamental group of a closed surface follows from Poincaré's polygon theorem (see e.g. \cite[Theorem 6.14]{Series}). The formula linking the genus $g$ with the area of $P$ is an application of the Gauss-Bonnet formula (see e.g. \cite[Theorem 3.3.11]{Labourie}).
\end{proof}

Note that $\Delta_{p,q,r}^+= \langle I_1I_2,I_2I_3 \rangle $ is an index 2 subgroup of $\Delta_{p,q,r}$ which preserves the orientation of $\mathbb{H}^2$, i.e., is contained in $\mathrm{PSL}_2(\mathbb{R})$. \\
Suppose that we are given a finite-sided convex polygon $P$ with no ideal vertices together with a set $S\subset \Delta_{p,q,r}^+$ of side pairings such that the angle sum corresponding to each vertex cycle is exactly $2\pi$. Denote by $\Gamma$ the group generated by $S$ and by $\Sigma=\mathbb{H}^2 / \Gamma$ the quotient hyperbolic surface. If $\Gamma$ is a finite-index subgroup of $\Delta_{p,q,r}$, the following lemma shows that we can associate to each infinite order element $w\in\Delta_{p,q,r}$ a closed geodesic on $\Sigma$. This closed geodesic is the projection to $\Sigma$ of the axis of $w$ for its action on $\mathbb{H}^2$.

\begin{lem}
Let $G$ be a group and $H$ be a subgroup of $G$ with index $n$. For all $g$ in $G$, there is an integer $k$ less than $n$, such that $g^k\in H$.
\end{lem}

\begin{proof}
Consider the $n+1$ left cosets of $H$ in $G$ given by $H, gH, g^2H,...,g^nH$. Since the index of $H$ in $G$ is $n$, two of these must be equal, say $g^aH$ and $g^bH$ where $a<b$. This implies that $g^{b-a}\in H$ where $b-a$ is an integer less than $n$.
\end{proof}

We want to determine the nature of the closed geodesic in $\Sigma$ associated with $I_1I_3I_2I_3$. In what follows, we exhibit a surface subgroup in $\Delta_{3,3,9}$ and prove that the closed geodesic associated with $I_1I_3I_2I_3$ is non-simple in Proposition~\ref{H2}.\\

We begin with a simple lemma. Let $f\in \mathrm{PSL}_2(\mathbb{R})$ be a hyperbolic isometry with axis $L_f$. Let $H_1$ and $H_2$ be two distinct hypercycles with base line $L_f$, and let $L_1$ and $L_2$ be two geodesics orthogonal to $L_f$. Let $A$, $B$, $C$ and $D$ denote the intersection points of $L_1$ with $H_1$, $L_1$ with $H_2$, $L_2$ with $H_2$ and $L_2$ with $H_1$ respectively (see Figure \ref{Figure3}). 

\begin{figure}[!h]
\centering
\includegraphics[width=0.5\textwidth]{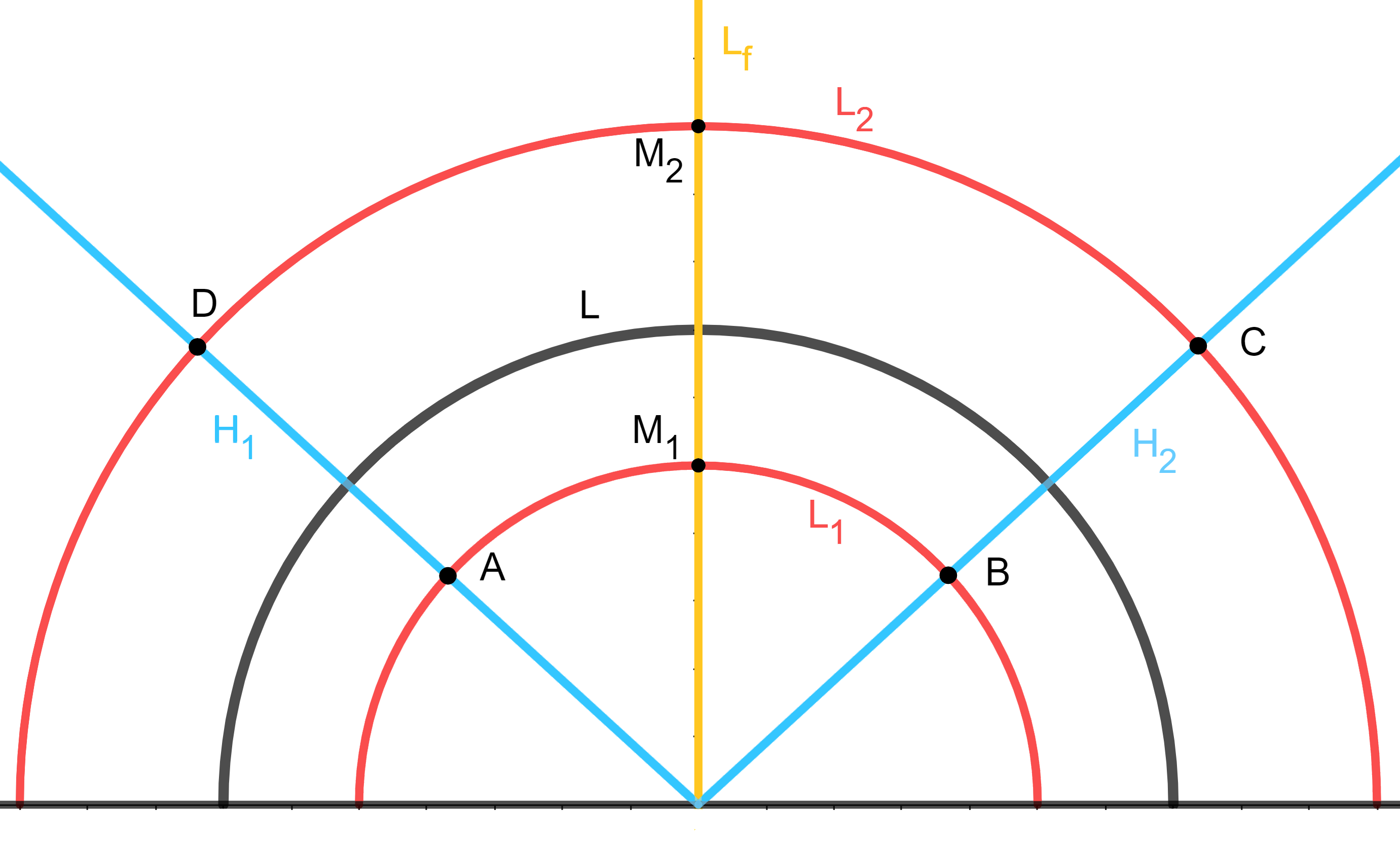}
\caption{The quadrilateral $ABCD$.}
\label{Figure3}
\end{figure}

\begin{samepage}
\begin{lem}\label{rectangle}
The following statements are equivalent:
\begin{enumerate}[label=\arabic*)]
\item $d(A,D)=d(B,C)$
\item $d_{H_1}(A,D)=d_{H_2}(B,C)$
\item $ABCD$ is invariant under the inversion $I_f$ in $L_f$.
\end{enumerate}
where $d_{H_1}$ and $d_{H_2}$ denote the arc length distances along $H_1$ and $H_2$ respectively.\\

Moreover, if one of the previous conditions is satisfied, then the diagonals of $ABCD$ intersect at their midpoints at a point of $L_f$.
\end{lem}
\end{samepage}

\begin{proof}
We work in the upper half-plane model of the hyperbolic plane. Without loss of generality, assume that $L_f$ is the vertical line based at $0$. In this model, geodesics orthogonal to $L_f$ are Euclidean circles centered in $0$, and hypercycles with base line $L_f$ are Euclidean lines passing through $0$. It follows that there are two positive real numbers $r_1$ and $r_2$ and angles $\theta_A, \theta_B \in (0,\pi)$ such that:
$$A=r_1e^{i\theta_A}, ~~~~B=r_1e^{i\theta_B}, ~~~~ C=r_2e^{i\theta_B},~~~~ D=r_2e^{i\theta_A}$$

Denote by $M_1$ and $M_2$ the intersection points of $L_f$ with $L_1$ and $L_2$, respectively:
$$M_1=r_1e^{i\frac{\pi}{2}}, ~~~~M_2=r_2e^{i\frac{\pi}{2}}$$

The arc length distances between two points along a hypercycle is the hyperbolic distance between the orthogonal projections of these points on the baseline, multiplied by the hyperbolic cosine of the hyperbolic distance between these points and the baseline (see \cite[p.68]{Lob}).

Therefore, we have: 
$$d_{H_1}(A,D)=d(M_1,M_2)\times \mathrm{cosh}\bigl(d(M_1,A)\bigl) ~~~~d_{H_2}(B,C)=d(M_1,M_2) \times \mathrm{cosh}\bigl(d(M_1,B)\bigl).$$
This shows the equivalence between 1) and 2).\\

For two points $z$ and $w$ in the upper half-plane model, the hyperbolic distance between $z$ and $w$ is given by: 
$$d(z,w)=\mathrm{cosh}^{-1}\left(1+\frac{|z-w|^2}{2\mathrm{Im}(z)\mathrm{Im}(w)}\right).$$

If $d(A,D)=d(B,C)$, then $\theta_A+\theta_B=\pi$, which in turn implies that $ABCD$ is invariant under the inversion $I_f$ in $L_f$. Conversely, if $ABCD$ is invariant under the inversion $I_f$ in $L_f$, then $d(A,D)=d(B,C)$ because $I_f$ is an isometry of $\mathbb{H}^2$. Hence 2) is equivalent to 3).\\

To see that if one of these conditions is satisfied, then its diagonals intersect at their midpoints, denote by $L$ the perpendicular bisector of the segment $[M_1M_2]$. Since $L$ is orthogonal to $L_f$, the inversion $I$ in $L$ preserves the hypercycles with base line $L_f$. Moreover, as the perpendicular bisector of $[M_1M_2]$, the inversion $I$ maps $L_1$ to $L_2$ and vice versa. Therefore, the quadrilateral $ABCD$ is also invariant under $I$, hence the intersection between $L$ and $L_f$ is a center of symmetry which is the intersection point of the diagonals.

\end{proof}

Note that the isometry $I_4=I_3I_2I_3\in \mathrm{PGL}_2(\mathbb{R})$ is the inversion in the geodesic $L_4=I_3(L_2)$. Since $L_1$ and $L_4$ are disjoint geodesics, $f=I_1I_3I_2I_3=I_1I_4$ is a hyperbolic isometry of $\mathbb{H}^2$ and its axis $L_f$ is the common perpendicular to $L_1$ and $L_4$ (see \cite[Theorem 7.38.1]{Beardon}). \\

We now focus on the specific triangular group $\Delta_{3,3,9}$.\\
Let $I_1$, $I_2$ and $I_3$ denote inversions in three concurrent geodesics $L_1, L_2$ and $L_3$, such that $\angle(L_2,L_3)=\angle(L_1,L_3)=\pi/3$ and $\angle(L_1,L_2)=\pi/9$, so that $\langle I_1, I_2,I_3\rangle= \Delta_{3,3,9}$. 
We denote by $O$ (resp. $A_1$ and $A_2$) the intersection point of $L_1$ with $L_2$ (resp. of $L_1$ with $L_3$ and of $L_2$ with $L_3$). The area of the hyperbolic triangle $OA_1A_2$ is: 

$$\pi-\left(\frac{\pi}{3}+\frac{\pi}{3}+\frac{\pi}{9} \right) = \frac{2\pi}{9}.$$

Notice that $I_2I_1\in \mathrm{PSL}_2(\mathbb{R})$ is an elliptic element of order 9 fixing $O$.\\

For every integer $k$ $mod$ $9$, we write $A_{2k+1}= (I_2I_1)^k(A_1)$ and $A_{2k+2}=(I_2I_1)^k(A_2)$ and consider the regular hyperbolic $18$-gon $P=A_1A_2...A_{18}$, with center $O$. It consists of $18$ isometric copies of the triangle $OA_1A_2$ hence its area is $18 \times 2\pi/9 = 4\pi$.\\

The following isometries in $\Delta_{3,3,9}^+$ together with their inverses form a set of side pairings for $P$ (see Figure \ref{Figure2}).
\begin{align*}
 s_I&=(I_2I_1)^2I_2I_3, & s_{II}&=(I_1I_2)^2s_I(I_2I_1)^2, & s_{III}&=I_2I_1s_I I_1I_2, \\
 s_{IV}&=(I_1I_2)^4I_3I_2I_1I_2, & s_V&=(I_1I_2)^3s_{IV}(I_2I_1)^3, & s_{VI}&=(I_2I_1)^3s_I(I_1I_2)^3, \\
 s_{VII}&=(I_2I_1)^4s_I(I_1I_2)^4, & s_{VIII}&=(I_2I_1)^3s_{IV}(I_1I_2)^3, & s_{IX}&=(I_1I_2)^3s_I(I_2I_1)^3.
\end{align*}

\begin{figure}[!h]
\centering
\includegraphics[width=0.6\textwidth]{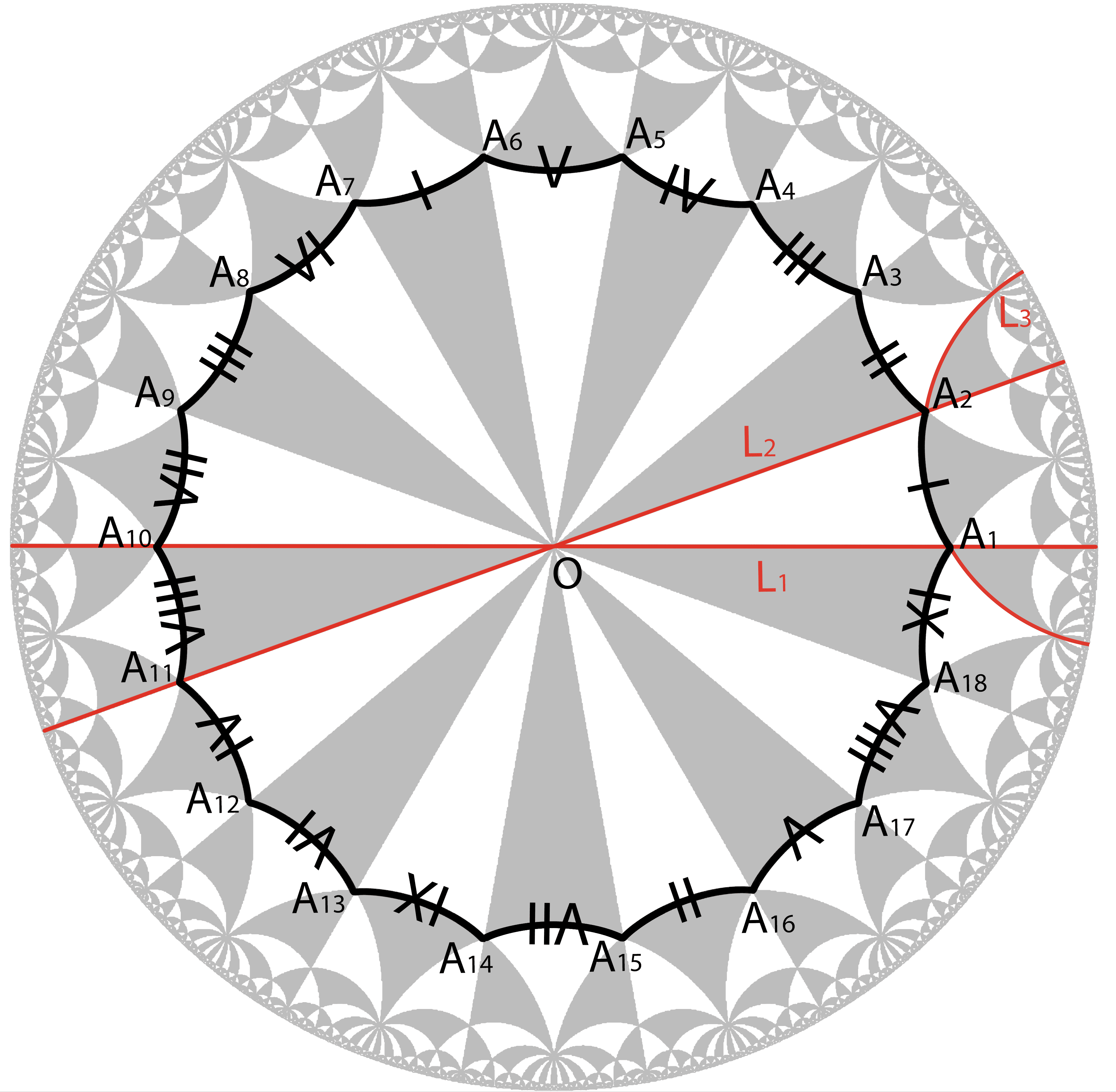}
\caption{The 18-gon $P$ with relevant side pairings}
\label{Figure2}
\end{figure}

There are exactly six vertex cycles, each of which have angle sum equal to $2\pi$. The group $\Gamma$ generated by $S=\{s_I^\pm,..., s_{IX}^\pm\}$ is isomorphic to the fundamental group of a closed oriented surface of genus 2 by Theorem \ref{Poincaré} and has index 18 in $\Delta_{3,3,9}$.
The group $D= \langle I_1,I_2 \rangle $ is isomorphic to a dihedral group of order 18 and acts simply transitively on the 18 triangles contained in $P$. 

We deduce the following right coset decomposition of $\Gamma$ in $\Delta_{3,3,9}$:

\begin{equation} \label{rd}
\Delta_{3,3,9} = \bigsqcup_{g\in D} \Gamma g 
\end{equation}

We denote by $\Sigma=\mathbb{H}^2 / \Gamma = P / \sim_S$ the quotient hyperbolic surface where $x\sim_S y$ if there is $s\in S$ such that $s(x)=y$.

\begin{prop}\label{H2}
The closed geodesic in $\Sigma$ associated with the element $f=I_1I_3I_2I_3$ is non-simple.
\end{prop}

\begin{proof}
Let $L_f$ denote the axis of $f$, and $I_f$ the inversion in $L_f$. Consider the hypercycles $H_1$ and $H_2$ with base line $L_f$, passing through the pairs $\bigl(A_1,f(A_1)\bigl)$ and $\bigl(A_2,f(A_2)\bigl)$ respectively. Recall that $L_4=I_3(L_2)$, and both $L_1$ and $L_4$ are orthogonal to $L_f$. Denote by $N_1$ and $N_2$ the intersection between $L_4$ and $H_1$ and between $L_1$ and $H_2$ respectively (see Figure \ref{Figure4}). Consider the quadrilateral $Q=A_1N_1A_2N_2$ determined by the geodesics $L_1$, $L_4$ and the hypercycles $H_1$, $H_2$. It follows from Lemma \ref{rectangle} that the diagonals of $Q$ intersect at their midpoints on a point of $L_f$. Hence, the geodesic $L_f$ intersects the side $[A_1A_2]$ exactly at its midpoint. The same argument shows that $L_f$ also intersects the side $[A_{18}A_1]$ at its midpoint.\\

In order to understand its image in $\Sigma$, we first look at the image of $L_f=L^1_f$ under the side pairing $S_I$ identifying $[A_1A_2]$ with $[A_6A_7]$. This new geodesic $L^2_f$ intersects the sides $[A_6A_7]$ and $[A_7A_8]$ exactly at their midpoints. We now consider $L^3_f$, the image of $L^2_f$ under the side pairing $S_{VI}$ identifying $[A_7A_8]$ with $[A_{12}A_{13}]$ and we continue this procedure until we find some integer $k$ such that $L^k_f=L_f$. It happens at the 18-th step, that is $L^{19}_f=L_f$. For each $i\in \{1,...,18\}$ the geodesic $L^i_f$ intersects adjacent sides of $P$ exactly at their midpoints. We deduce that the projection of $L_f$ to $\Sigma$ yields a closed geodesic with $9$ self-intersections (see Figure \ref{Figure5}).
\end{proof}

\begin{figure}[p]
\centering 
\begin{subfigure}[b]{0.9\textwidth}
        \centering
        \includegraphics[width=\textwidth,height=0.45\textheight,keepaspectratio]{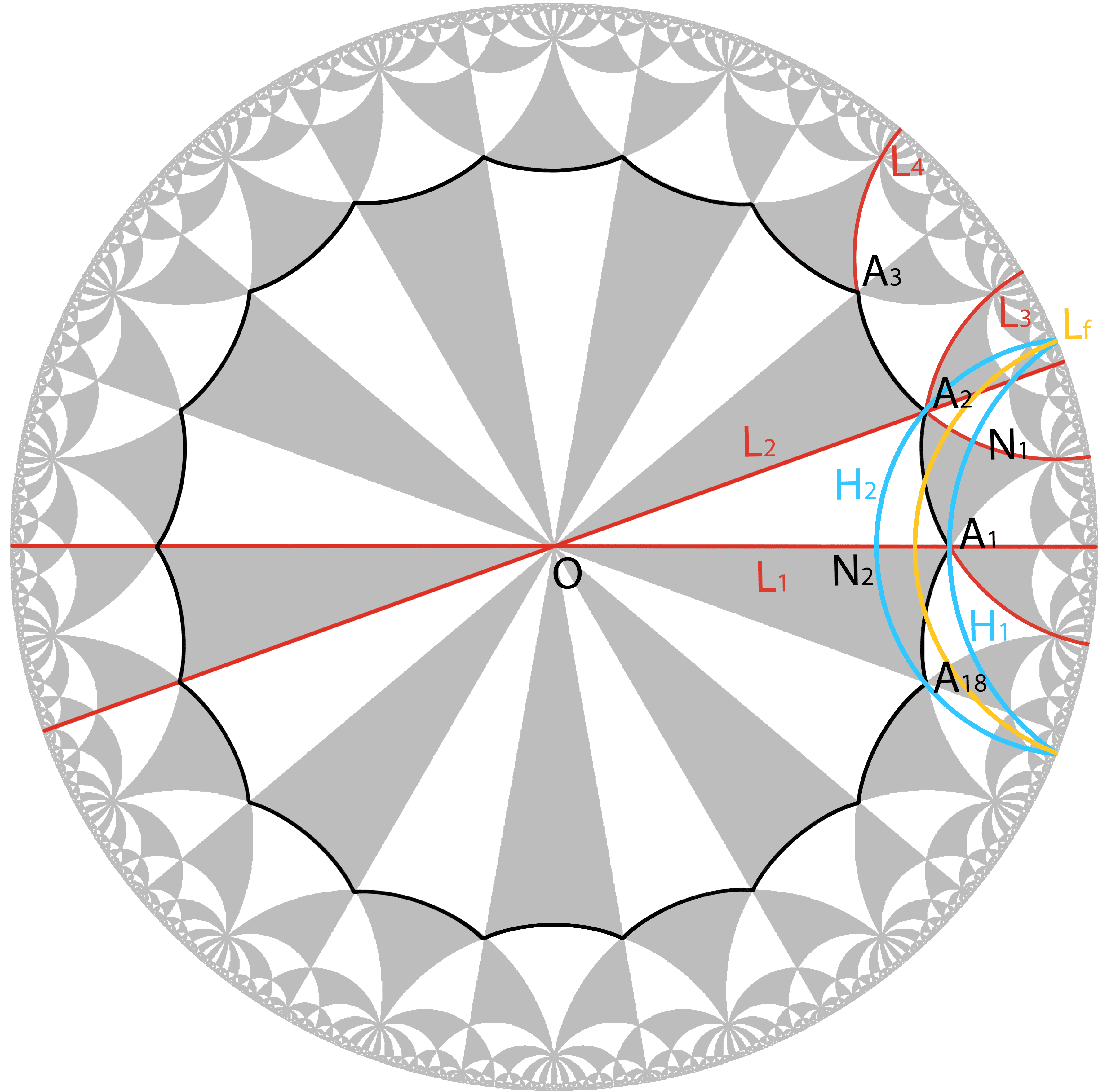}
		\caption{}
		\label{Figure4}
\end{subfigure}
\vfill
\begin{subfigure}[b]{0.9\textwidth}
        \centering
        \includegraphics[width=\textwidth,height=0.45\textheight,keepaspectratio]{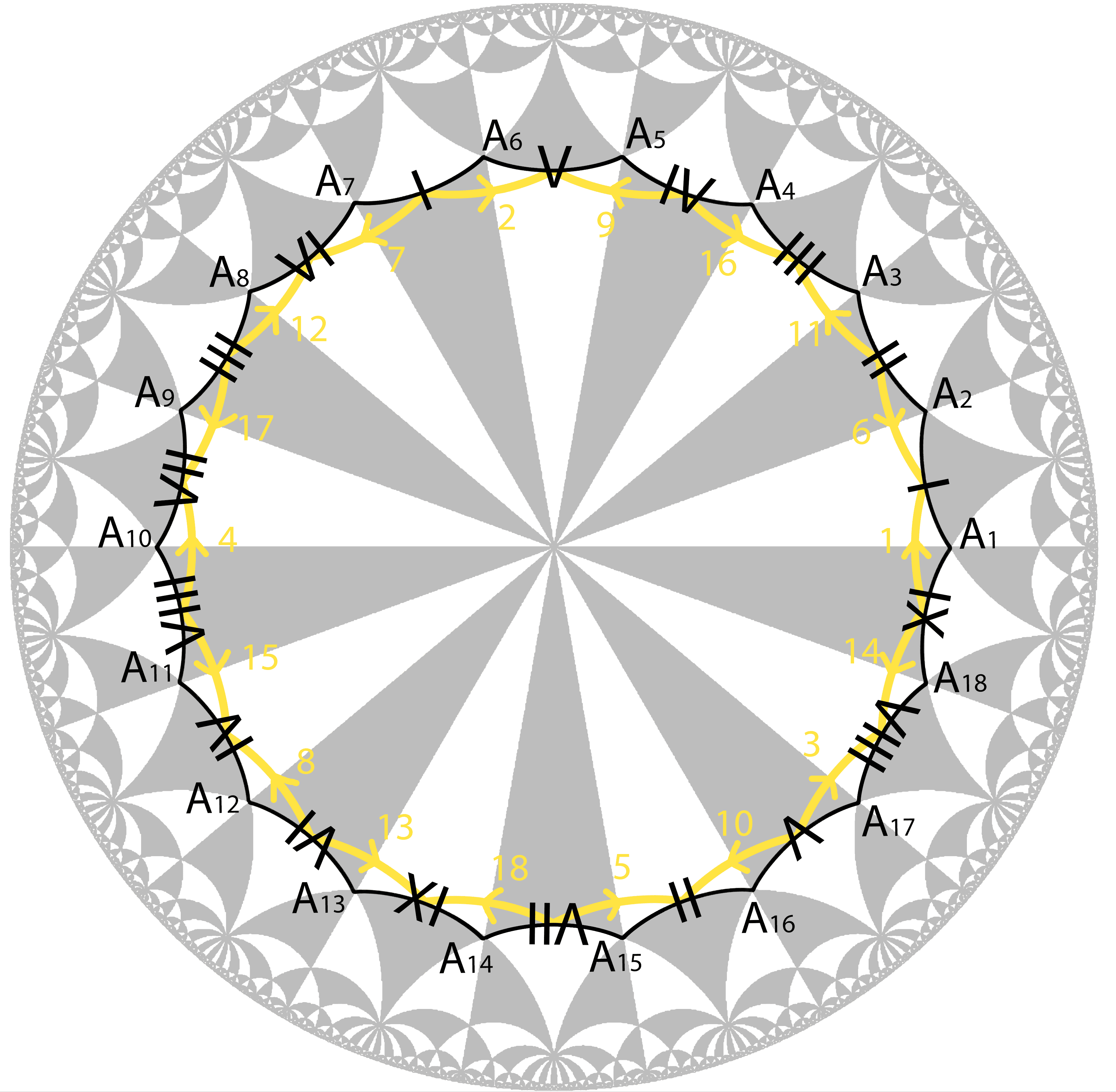}
		\caption{}
		\label{Figure5}
\end{subfigure}
\caption{The geodesic $L_f$ and its image in the genus 2 surface $\Sigma=P/\sim_S$}
\end{figure}

The important symmetry of the projection of $L_f$ to $\Sigma$ also gives us the following:

\begin{prop}\label{H3}
The $\Gamma$-orbit and $\Delta_{3,3,9}$-orbit of $L_f$ coincide:
$$\Delta_{3,3,9}(L_f)=\Gamma (L_f).$$
\end{prop}

\begin{proof}
Recall the right coset decomposition of $\Delta$ given in (\ref{rd}) :

\begin{equation*} 
\Delta_{3,3,9} = \bigsqcup_{g\in D} \Gamma g 
\end{equation*}

where $D=\langle I_1,I_2\rangle$ is the dihedral group preserving $P$. For all $w\in \Delta$, there are $g\in D$ and $\gamma\in \Gamma$ such that $w=\gamma g$. The geodesic $L_f$ passes through the midpoints of two adjacent sides of $P$. It follows that $g(L_f)$ also passes through the midpoints of two adjacent sides of $P$. Using the same notation that in the proof of Proposition \ref{H2}, this implies that there is $i\in \{1,...,18\}$ such that $g(L_f)=L^i_f$. Hence, there is $\gamma'\in \Gamma$ such that $g(L_f)=\gamma'(L_f)$ from which we deduce:
$$w(L_f)=\gamma g(L_f)=\gamma \gamma'(L_f) \subset \Gamma (L_f).$$
\end{proof}

\section{Simple-stable representations of $\Gamma_g$ in $\mathrm{PU}(2,1)$}\label{S}

In this last section, we first prove Theorem \ref{SSD} for genus 2 surfaces and then prove it for higher genus.

\subsection{Genus 2 surfaces}

We are ready to prove:

\begin{thm}\label{SS}
Let $\Gamma_2$ denote the fundamental group of a closed surface of genus $2$.\\ 
There exists simple-stable representations of $\Gamma_2$ in $\mathrm{PU}(2,1)$ which are not convex-cocompact. 
\end{thm}

\begin{proof}
Recall first that according to Proposition \ref{final}, there exists $\alpha_0=\alpha_9^0\in [0,\pi]$ such that the $(3,3,9;\alpha_0)$ representation of $\Delta_{3,3,9}$ in $\mathrm{PU}(2,1)$ is well-defined, discrete, faithful and geometrically finite with $I_1I_3I_2I_3$ parabolic. We denote by $\rho_0\in \mathcal{R}\bigl(\Delta_{3,3,9},\mathrm{PU}(2,1)\bigl)$ this representation. Recall that Proposition \ref{final} also implies that there is a unique $\rho_0(\Delta_{3,3,9})$-orbit of parabolic fixed point in the limit set $\Lambda\bigl(\rho_0(\Delta_{3,3,9})\bigl)$.\\

Denote by $r_0\in \mathcal{R}\bigl(\Gamma_2,\mathrm{PU}(2,1)\bigl)$ the restriction of $\rho_0$ to its surface subgroup defined in Section \ref{surfsubgrp}. The representation $r_0$ is also discrete, faithful and geometrically finite. By Proposition \ref{H3} and the existence of the boundary map given by Theorem \ref{b-map}, there is also a unique $r_0(\Gamma_2)$-orbit of parabolic fixed points in the limit set $\Lambda(r_0(\Gamma_2))$. Combining this with Lemma \ref{cyclic}, we deduce that the stabilizer in $r_0(\Gamma_2)$ of every parabolic fixed point is a cyclic subgroup conjugate (in $r_0(\Gamma_2)$) to $\langle I_1I_3I_2I_3\rangle \cap r_0(\Gamma_2)$. Since the element of $\Gamma_2$ associated to $I_1I_3I_2I_3$ is not simple by Proposition \ref{H2}, we deduce from Theorem \ref{Crit} that $r_0\in \mathcal{R}\bigl(\Gamma_2,\mathrm{PU}(2,1)\bigl)$ is a simple-stable representation. Moreover, $r_0(\Gamma_2)$ contains parabolic elements, hence $r_0$ cannot be convex-cocompact.
\end{proof}

\subsection{From genus 2 to higher genus}

In this last part, we explain how to deduce Theorem \ref{SSD} from Theorem \ref{SS}. The idea is to consider the restriction of the simple-stable representation $r_0\in S_{\mathcal{S}}\bigl(\Gamma_2,\mathrm{PU}(2,1)\bigl)$ defined in the previous section, to relevant finite index subgroups of $\Gamma_2$ isomorphic to $\Gamma_g$ for $g\geqslant 2$.\\

Let $\Sigma_2$ be a closed orientable surface of genus 2, $\beta$ be a non-intersecting and non-separating closed curve on $\Sigma_2$ and $g\geqslant 2$ be an integer.\\

By cutting $\Sigma_2$ along $\beta$ and gluing $(g-1)$ copies of the resulting two holed torus in a cyclic way, we obtain a closed surface $\Sigma_g$ of genus $g$. This gives a finite cyclic cover $p^{\beta}_g:\Sigma_g \to \Sigma _2$ of degree $g-1$ as described in Figure \ref{Figure6}.

\begin{figure}[!h]
\centering
\includegraphics[width=0.72\textwidth]{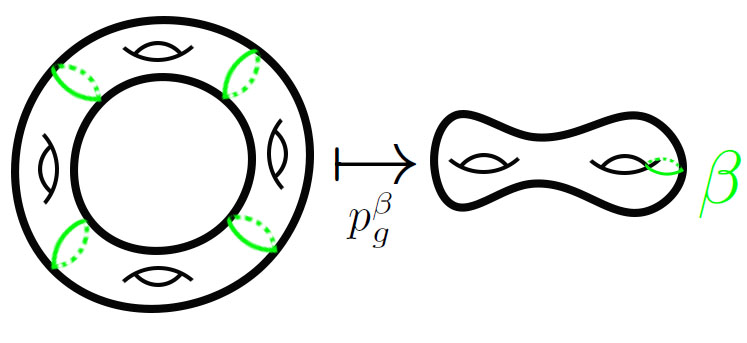}
\caption{The cyclic cover $p^{\beta}_g$ when $g=5$.}
\label{Figure6}
\end{figure}

We can also describe $p^{\beta}_g$ algebraically. Denote by $\hat{\imath}(\beta,\gamma)$ the algebraic intersection number between the closed curves $\beta$ and $\gamma$. The subgroup of $\pi_1(\Sigma_2)$ associated to $p^{\beta}_g$ consists of free homotopy classes of closed curves on $\Sigma_2$ whose lifts are closed curves on $\Sigma_g$. These elements are precisely the elements of the kernel $\ker(\Psi^{\beta}_g)$ where

$$\begin{array}{cccccc}
\Psi^{\beta}_g &: & \pi_1(\Sigma_2) & \to & \mathbb{Z}/(g-1)\mathbb{Z} & \\
 & & \gamma & \mapsto & \hat{\imath}(\beta, \gamma) & \mathrm{mod} \quad (g-1).\\
\end{array}
$$

We can now deduce:

\SSD*

\begin{proof}
Denote by $r_0\in \mathcal{R}\bigl(\Gamma_2,\mathrm{PU}(2,1)\bigl)$ the discrete, faithful, geometrically finite and simple-stable representation given by Theorem \ref{SS}. We also denote by $\gamma_0 \in \Gamma_2$ a non simple element of $\Gamma_2$ with $r_0(\gamma_0)$ parabolic and by $L$ the closed geodesic of $\Sigma_2$ associated with $\gamma_0$ (see Figure \ref{Figure5}).\\

Let $\beta$ be a non-separating simple closed curve on $\Sigma_2$ satisfying the following conditions:

\begin{enumerate}[label=\arabic*)]
\item The algebraic intersection number $\hat{\imath}(\beta,L)$ is equal to 0,
\item There exists a subpath $\widetilde{L}: [a,b] \to \Sigma_2$ of $L$ such that:
\begin{enumerate}[label=\roman*)]
\item $\widetilde{L}(a)\in \beta$ and $\widetilde{L}(b)\in \beta$,
\item For all $t\in ]a,b[$, $\widetilde{L}(t) \notin \beta$,
\item The restriction of $\widetilde{L}$ to $]a,b[$ is not injective.
\end{enumerate} 
\end{enumerate} 

An example of such a closed curve $\beta$ is shown in Figure \ref{Figure7}.\\

Let $g\geqslant 2$ and define $p=p^{\beta}_g:\Sigma_g \to \Sigma_2$ to be the finite cyclic cover of degree $(g-1)$ defined above. The hyperbolic structure on $\Sigma_2$ lifts by $p$ to a hyperbolic structure on $\Sigma_g$. The closed geodesic $L$ admits $g-1$ lifts to $\Sigma_g$ by $p$. Each of these lifts is a geodesic closed curve because $L$ is a geodesic and $\hat{\imath}(\beta,L)=0$. Moreover, the second condition above on $\beta$ guarantees that these lifts are self-intersecting, hence their free homotopy classes correspond to non-simple elements of $\pi_1(\Sigma_g)$ by Proposition \ref{malin}.\\ 

Define $$\begin{array}{cccccc}
\Psi &: & \pi_1(\Sigma_2) & \to & \mathbb{Z}/(g-1)\mathbb{Z} & \\
 & & \gamma & \mapsto &  \hat{\imath}(\beta, \gamma) & \mathrm{mod} \quad (g-1).\\
\end{array}
$$

The restriction $r$ of $r_0$ to $\ker(\Psi)$ is a discrete, faithful and geometrically finite representation of $\Gamma_g$ in $\mathrm{PU}(2,1)$. The representation $r$ is not convex-cocompact because $r(\Gamma_g)$ contains parabolic elements. Every parabolic element contained in $r(\Gamma_g)$ is the image of a non-simple element of $\Gamma_g$, therefore, $r$ is simple-stable by Theorem \ref{Crit}.\\

Since the set of conjugacy classes of simple-stable representations is a domain of discontinuity by Proposition \ref{DD}, we deduce that $S_{\mathcal{S}}\bigl(\Gamma_g,\mathrm{PU}(2,1)\bigl)$ is a discontinuity domain strictly larger than $\mathrm{CC}\bigl(\Gamma_g,\mathrm{PU}(2,1)\bigl)$.\\

Let $\alpha^{\min}_9$ be as defined in Proposition \ref{final}. For $\alpha\in [\alpha^{\min}_9, \pi]$, denote by $r_\alpha$ the restriction of the $(3,3,9;\alpha)$ representation to its finite index surface subgroup $\ker(\Psi)$ so that:

$$\lim\limits_{\substack{\alpha \to \alpha_0 \\ \alpha>\alpha_0}} r_\alpha = r.$$

It follows from Proposition \ref{final} that $r$ arises as a limit of convex-cocompact representations. Moreover, since the set $S_{\mathcal{S}}\bigl(\Gamma_g,\mathrm{PU}(2,1)\bigl)$ of conjugacy classes of simple-stable representations is open, it also contains conjugacy classes of some $r_\alpha$ for $\alpha\leqslant \alpha_0$. This implies that $S_{\mathcal{S}}\bigl(\Gamma_g,\mathrm{PU}(2,1)\bigl)$ also contains conjugacy classes of representations which are not discrete and faithful by Proposition \ref{final}.
\end{proof}

\begin{figure}[!h]
\centering
\includegraphics[width=0.7\textwidth]{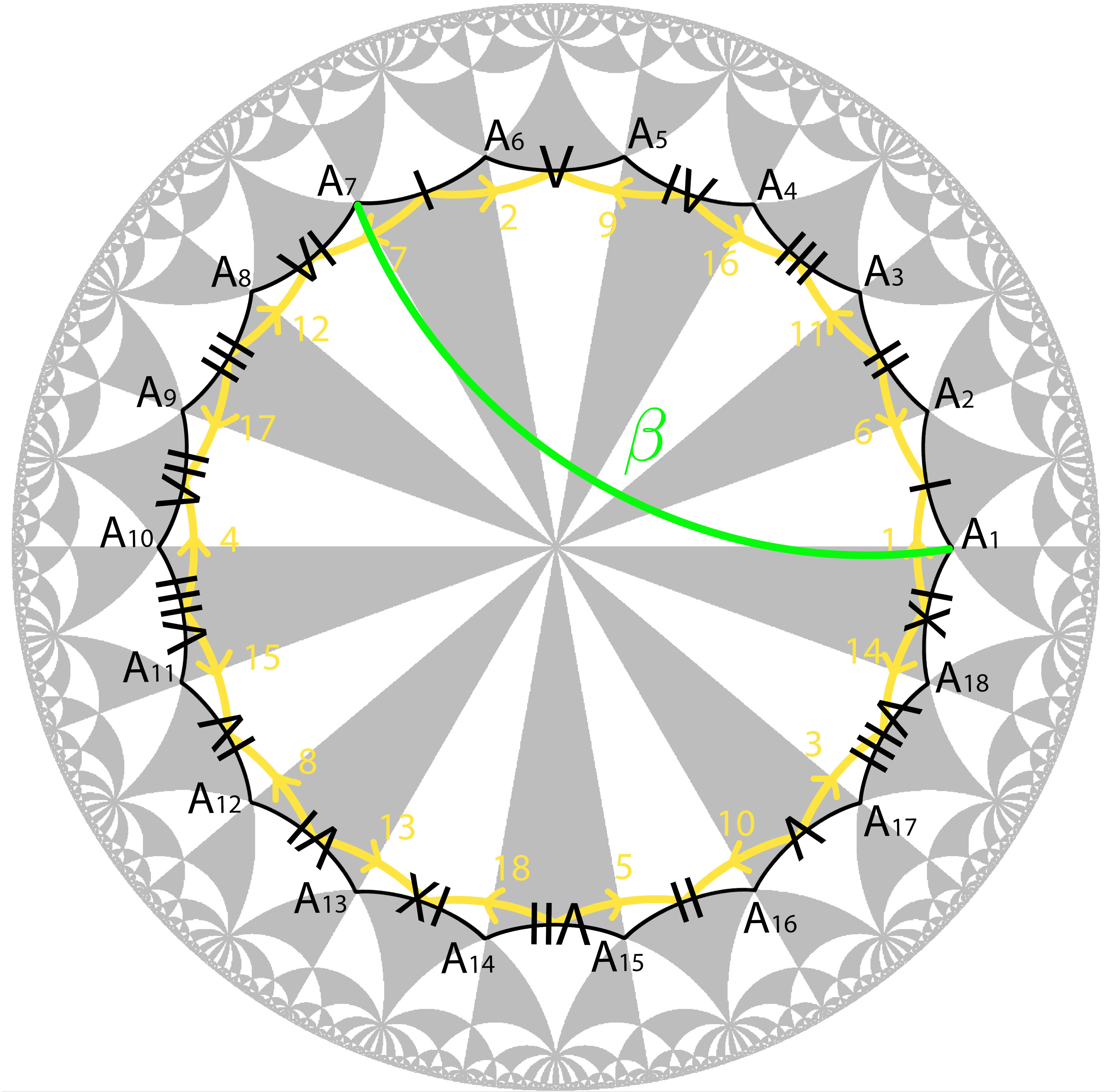}
\caption{An example of a curve $\beta$.}
\label{Figure7}
\end{figure}

\newpage

\bibliographystyle{plain}

\bibliography{Biblio.bib}

\end{document}